\documentclass[11pt]{article}
\usepackage{graphicx} 

\usepackage{authblk}
\usepackage[margin=2.3cm]{geometry}
\usepackage{t1enc}
\usepackage[utf8]{inputenc}
\usepackage{amsmath,amsthm,amssymb}
\usepackage{graphicx}
\usepackage{enumerate}
\usepackage{hyperref}
\usepackage{bm}
\usepackage{comment}
\usepackage{amsfonts}
\usepackage{graphicx,caption}
\usepackage{bm}
\usepackage{amsmath, amsthm, amssymb}
\usepackage{graphicx}
\usepackage{hyperref}
\usepackage{relsize}
\usepackage{blkarray}
\usepackage[table]{xcolor}
\usepackage{tikz}
\usetikzlibrary{decorations.pathreplacing}
\usepackage{tabstackengine}

\stackMath

\usepackage{subfigure}

\usepackage{algpseudocode}

\usepackage{dsfont}
\newcommand{\mathbbm}[1]{\mathds{#1}}

\theoremstyle{plain}
\usepackage{amsthm}
\makeatletter
\newcommand{\newreptheorem}[2]{\newtheorem*{rep@#1}{\rep@title}\newenvironment{rep#1}[1]{\def\rep@title{#2 \ref*{##1}}\begin{rep@#1}}{\end{rep@#1}}}
\makeatother

\newtheorem{theorem}{Theorem}
\newtheorem*{theorem-non}{Theorem}
\newtheorem*{non-lemma}{Lemma}
\newtheorem{lemma}[theorem]{Lemma}
\newreptheorem{lemma}{Lemma}

\newtheorem{corollary}[theorem]{Corollary}

\theoremstyle{definition}
\newtheorem{remark}[theorem]{Remark}
\newtheorem{definition}[theorem]{Definition}

\DeclareMathOperator{\rang}{rank}

\DeclareMathOperator{\Tr}{Tr}

\DeclareMathOperator{\tors}{\mathrm{tor}}

\DeclareMathOperator{\dkl}{D_{KL}}

\title{Product approximations for uniform spanning trees and determinantal processes}
\author[1]{Andr\'as M\'esz\'aros}
\date{}
\affil[1]{HUN-REN Alfr\'ed R\'enyi Institute of Mathematics, Budapest}

\author[2]{Yuval Peled}
\affil[2]{Einstein Institute of Mathematics, Hebrew University of Jerusalem}

\begin{document}

\maketitle
    \begin{abstract}
We prove that for every finite connected simple graph $G$ on $n$ vertices with minimum degree $k\geq 2$, there is a coupling of its uniform spanning tree and its random $1$-out subgraph under which the expected number of differing edges is $O(n\log (k)/\sqrt{k})$. 

More generally, we establish analogous couplings between a wide class of determinantal processes and their corresponding $1$-out processes, including Kalai's determinantal hypertrees, with expected symmetric difference negligible compared to the size of the determinantal set.
While the graphical case admits a surprisingly short and elegant proof, the more general result requires establishing an asymptotic equipartition property for the determinantal measures in this class, from which we also derive exponential growth rates for the homology torsion of determinantal hyperforests in higher-dimensional regular high-degree complexes.
\end{abstract}
\section{Introduction}

\subsection{Uniform spanning trees and $1$-out graphs}
Spanning trees are fundamental objects in discrete mathematics and computer science,  
and the uniform spanning tree (UST) of a connected graph $G=(V,E)$ is one of the most studied objects in discrete probability (see ~\cite{lyons2017probability}).
Our main result shows that if the minimum degree of $G$ tends to infinity, 
then its UST can be coupled with a random subgraph generated by a product measure, called the $1$-out graph (see~\cite{frieze2017random}), so that their 
edge sets differ by $o(|V|)$ edges in expectation. 

To sample the $1$-out graph, we let each $v\in V$ independently draw a uniform random neighbor $R(v)$ of $v$, and let $G^{1o}$ be the spanning subgraph of $G$ 
whose edge set is
\(
    E(G^{1o})=\big\{\{v,R(v)\}\,:\,v\in V\big\}\,.
\)
Thus, $G^{1o}$ is determined by independent choices at the vertices of $G$.

There are several indications in the literature of a connection between the UST and the random $1$-out graph. First, Joyal's proof of Cayley's formula~\cite{joyal1981theorie} gives an $n^2$-to-one map from functions $f:[n]\to[n]$ onto the spanning trees of $K_n$.
Thus a uniform random function yields a UST of $K_n$. The construction preserves all edges $\{v,f(v)\}$ for which $v$ does not lie on a directed cycle of $f$. 
Since a uniform random function $f$ has $O(\sqrt{n})$ cyclic vertices in expectation, the resulting tree differs from the undirected graph of $f$ by $O_{\mathbb{P}}(\sqrt{n})$ edges (see also~\cite[Lemma~6]{heilman2022tree}). On the other hand, resampling the fixed points of $f$ to obtain uniform choices from $[n]\setminus\{v\}$ changes only $O(1)$ additional edges in expectation. Hence, Joyal's construction gives a coupling of the UST of $K_n$ and $K_n^{1o}$ with $O(\sqrt{n})$ expected edge differences.

For an arbitrary graph $G$, there is no evident analogue of Joyal's correspondence. Nevertheless, we show that for graphs with large minimum degree, the UST and the random $1$-out graph can be coupled so that their edge sets are close in expectation.

\begin{theorem}\label{thm1dcoupling}
    Let $G$ be a finite simple connected graph on $n$ vertices with minimum degree $k\geq 2$, and let $T$ be a uniform spanning tree of $G$.
    There is a coupling of $T$ and $G^{1o}$ in which
    \[
        \mathbb{E}\left[\frac{|E(G^{1o})\Delta E(T)|}{n}\right]
        \leq O\left(\frac{\log k}{\sqrt{k}}\right)\,.
    \]
\end{theorem}

Equivalently, the $1$-Wasserstein distance between the laws of $T$ and $G^{1o}$, with respect to the normalized Hamming distance $d(H_1,H_2)=n^{-1}|E(H_1)\Delta E(H_2)|$ on spanning subgraphs of $G$, is $O(\log k/\sqrt{k})$.

We note that the proof of Theorem \ref{thm1dcoupling} is, perhaps surprisingly, technically simple. The coupling follows by combining a lower bound of Kostochka~\cite{kostochka1995number} on the number of spanning trees in a graph in terms of its degree sequence with Marton's coupling lemma~\cite{marton1986simple}. 

Using a second moment argument, we also bound the proportion of vertices whose neighborhoods differ under the same coupling.
For a graph $H$, a vertex $o\in V(H)$ and an integer $r\geq 0$, we write $B_H(o,r)$ for the subgraph of $H$ induced by the vertices of graph distance at most $r$ from $o$, rooted at $o$. 
In addition, when $H$ and $H'$ have the same vertex set, equality of these neighborhoods is understood as equality of their vertex and edge sets, with the same root.

\begin{theorem}\label{thm:local-coupling}
    Under the assumptions of Theorem~\ref{thm1dcoupling}, the coupling may be chosen so that, for every fixed~$r$,
    \[
        \mathbb{E}\left[\frac{|\{o\in V(G)\,:\,B_T(o,r)\neq B_{G^{1o}}(o,r)\}|}{n} \right] 
        \leq 10\varrho\cdot\left(\frac{C\log k}{\sqrt{k}}\right)^{2^{-r}}\,,
    \]
    where
    \[
        \varrho=\frac{1}{n}\sum_{o\in V(G)}\left(\frac{\deg_G(o)}{k}\right)^2
    \]
    measures the irregularity of $G$, and $C$ is a universal constant. 
\end{theorem}

Another well-studied connection between the UST and the random $1$-out graph is their common local weak limit.  
Here, local weak convergence means convergence in distribution of the rooted neighborhoods of every fixed radius, viewed up to root preserving isomorphism (see Section~\ref{SecRootedGraphs}).
Grimmett~\cite{grimmett1980random} showed that the local weak limit of the UST of $K_n$ is the Poisson$(1)$ Galton--Watson tree conditioned to survive forever.  
More explicitly,  this limit is obtained from a one-sided infinite path, rooted at its endpoint, by identifying each vertex of the path with the root of an independent unconditioned Poisson$(1)$ Galton--Watson tree. 
It is easy to see that this is also the local weak limit of $K_n^{1o}$, see~\cite{linial2019enumeration}. In the $1$-out model, the path corresponds to the successive choices of the root, its chosen neighbor and so on, while the trees ``hanging'' on its vertices correspond to vertices whose successive choices lead to the path.
Nachmias and Peres~\cite{nachmias2022local} extended Grimmett's result to sequences of high degree almost regular graphs. In the regular case, $\varrho=1$, and a direct exploration of the random $1$-out subgraph gives the same local weak limit. Hence, Theorem~\ref{thm:local-coupling} directly recovers their result in this case.

\begin{corollary}
    Let $G_n$ be a $d_n$-regular connected graph such that $d_n\to\infty$, and let $T_n$ be a uniform random spanning tree of $G_n$. Then the local weak limit of $T_n$ is the Poisson(1) branching process conditioned to survive forever.
\end{corollary}

\subsection{Determinantal and $1$-out processes}\label{sec:det_intro}

The coupling of Theorem~\ref{thm1dcoupling} extends from uniform spanning trees to a more general class of determinantal processes in the setting introduced by Nachmias and the second author~\cite{nachmias2025local}.

Let $\mathfrak{G}_n=(V_n,U_n,E_n)$ be a sequence of simple bipartite graphs. We will often drop the index $n$. We assume that $\mathfrak{G}$ is $C_4$-free, $|V|=n$ and $|U|=m$, and that the degree of each $v\in V$ is $d$ and the degree of each $u\in U$ is $k+1$. We assume that $k\geq 1$ is fixed and $d\to\infty$ as $n\to\infty$.

Let $B=B_n$ be an $n\times m$ signed adjacency matrix of $\mathfrak{G}$, that is, $B_{v,u}\in\{-1,+1\}$ for all $vu\in E$, and $B_{v,u}=0$ otherwise. Let $H=H_n\subset\mathbb{R}^m$ be the row space of $B$, and let $P_H$ be the matrix of the orthogonal projection onto $H$ in the standard basis of $\mathbb{R}^m$. Write $r=\dim H$. For a subset $\mathcal{T}\subset U$, let $P_H\restriction\mathcal{T}$ denote the submatrix of $P_H$ whose rows and columns are indexed by $\mathcal{T}$. The determinantal process associated with $H$ is a probability measure $\mathbb{P}^H$ on subsets $\mathcal{T}\subset U$ of size $r$, given by
\[
\mathbb{P}^H(\mathcal{T})=\det(P_H\restriction\mathcal{T})\,.
\]
As discussed below, this setting includes several well-studied random processes, such as uniform spanning trees of regular graphs, Kalai's determinantal hypertrees~\cite{kalai1983enumeration}, determinantal hyperforests in high degree regular polytopal complexes, and more. See~\cite{nachmias2025local} and the references therein for further examples.

Additionally, there is a natural $1$-out process in this setting. Independently for each $v\in V$, draw a uniform random neighbor $R(v)$ of $v$, and set
\(
    \mathcal{T}^{1o}=\{R(v)\,:\,v\in V\}\subset U\,.
\)
For a subset $\mathcal{S}\subset U_n$, write $\mathfrak{G}_n[\mathcal{S}]$ for the subgraph of $\mathfrak{G}_n$ induced by $V_n\cup\mathcal{S}$. Let $\mathcal{T}_n$ be drawn from $\mathbb{P}^{H_n}$, and let $o_n$ be a uniform random vertex of $V_n$, chosen independently of both processes. The main result in~\cite{nachmias2025local} identifies the local weak limit of $(\mathfrak{G}_n[\mathcal{T}_n],o_n)$, which is also the local weak limit of the corresponding $1$-out graphs $(\mathfrak{G}_n[\mathcal{T}_n^{1o}],o_n)$.
Here we show that the two processes can also be coupled so that their symmetric difference has $o(n)$ elements in expectation.

\begin{theorem}\label{thmgeneralcoupling}
    In the setup above, for each $n$ there is a coupling of $\mathcal{T}_n$ and $\mathcal{T}_n^{1o}$ such that
    \[
        \lim_{n\to\infty}\mathbb{E}\left[\frac{|\mathcal{T}_n\Delta\mathcal{T}_n^{1o}|}{n}\right]=0\,.
    \]
\end{theorem}

Unlike in the UST case, the proof of Theorem~\ref{thmgeneralcoupling} uses the local weak limit of the determinantal process established in~\cite{nachmias2025local}, so it does not give an alternative proof of that local limit.

\medskip

One of the key ingredients in the proof of Theorem~\ref{thmgeneralcoupling} is the following asymptotic equipartition property of the determinantal measure $\mathbb{P}^{H_n}$.
The constant $c_k$ appearing in the theorems below was introduced by the first author~\cite{meszaros2026homology} in his study of the homology torsion growth of determinantal hypertrees. 
We recall its definition in Definition~\ref{ckdef}, and note here that it depends only on $k$ and not on the matrices $B_n$. We have $c_1=0$, while $c_k>0$ for $k\geq2$.

\begin{theorem}\label{thm1}
    In the setup above, let $\mathcal{T}_n$ be drawn from the measure $\mathbb{P}^{H_n}$. Then, as $n\to\infty$,
    \[
        \frac{\log\mathbb{P}^{H_n}(\mathcal{T}_n)}{n}+\log(d)
        \text{ converges to }2c_k\text{ in probability.}
    \]
\end{theorem}

In addition to Theorem~\ref{thmgeneralcoupling}, we find several further applications of Theorem~\ref{thm1}. 
First, the following immediate corollary regarding the number $\tau(G)$ of spanning trees of a high-degree regular graph $G$ was proved by Alon~\cite{alon1990number}.
\begin{corollary}\label{CorNumberOfTrees}
Let $G_n$ be a simple $d=d_n$-regular connected graph on $n$ vertices. If $d\to\infty$ as $n\to\infty$ then
\[
\frac{\log(\tau(G_n))}n-\log d\to 0\,.
\]
\end{corollary}

Second, we determine the asymptotics of the Shannon-entropy $H(\mathcal{T}_n)$ of the determinantal law.

\begin{corollary}\label{corEntropy}
    Let $\mathcal{T}_n$ be drawn from the measure $\mathbb{P}^{H_n}$. Then, as $n\to\infty$,
    \[
        \frac{H(\mathcal{T}_n)}n-\log(d)\to-2c_k\,.
    \]
\end{corollary}

\medskip
\subsection{Homology torsion growth in hypertrees and hyperforests}\label{SecTorsionIntroduction}
For the third application of Theorem~\ref{thm1}, we consider cell complexes and their determinantal hyperforests, which form one of the
main classes of examples in our general setting.

A cell complex is a topological space obtained by attaching balls, called cells, along their boundaries.
For a cell complex $\Delta$, write $F_i(\Delta)$ for its set of $i$-dimensional cells, $f_i(\Delta)=|F_i(\Delta)|$ and 
$H_i(\Delta,A)$ for its $i$-th homology group with coefficients in an abelian group $A$. 

Suppose that $\Delta$ is a finite $\ell$-dimensional cell complex, with $\ell\geq1$. A subcomplex $S\subset\Delta$ is called a \emph{spanning hyperforest} 
if it contains the $(\ell-1)$-skeleton of $\Delta$, satisfies $H_\ell(S,\mathbb{Q})=0$, and the inclusion $S\subset\Delta$ induces an isomorphism from $H_{\ell-1}(S,\mathbb{Q})$ to $H_{\ell-1}(\Delta,\mathbb{Q})$. 
Equivalently, its $\ell$-cells index a basis of the column space of the $\ell$-dimensional boundary matrix $B=\partial_\ell$. 
We write $\mathcal{S}_\ell(\Delta)$ for the set of spanning $\ell$-hyperforests of $\Delta$. For background on these higher dimensional counterparts of spanning trees, see the survey of Duval, Klivans and Martin~\cite{DKM} and the references therein.
The determinantal measure associated with the row space $H$ of $B$ is supported on the sets $F_\ell(S)$ with $S\in\mathcal{S}_\ell(\Delta)$. 
The corresponding random subcomplex is called a \emph{determinantal hyperforest} of $\Delta$.
Lyons~\cite[Proposition~3.1]{lyons2009random} showed that this measure satisfies
\[
    \mathbb{P}^H(F_\ell(S))=\frac{t_{\ell-1}(S)^2}{\tau_\ell(\Delta)}\,,
\]
where
\[
    t_{\ell-1}(S)=|\tors H_{\ell-1}(S,\mathbb{Z})|,
    \qquad
    \tau_\ell(\Delta)=\sum_{S\in\mathcal{S}_\ell(\Delta)}t_{\ell-1}(S)^2\,,
\]
and $\tors A$ denotes the torsion subgroup of an abelian group $A$.

Fix integers $\ell,k\geq1$, and let $(\Delta_j)$ be a sequence of finite cell complexes of dimension $\ell$ satisfying:
\begin{enumerate}[(1)]
    \item\label{polytop1} The complex is polytopal: its closed cells are polytopes, and the intersection of any two closed cells is a face of each of them;
    \item\label{polytop2} Every $\ell$-cell has precisely $k+1$ facets of dimension $\ell-1$;
    \item Every $(\ell-1)$-cell is contained in precisely $d=d_j$ cells of dimension $\ell$, where $d_j\to\infty$;
    \item\label{polytop4} No pair of distinct $(\ell-1)$-cells is contained in more than one $\ell$-cell.
\end{enumerate}
After choosing orientations of the cells, the boundary matrix $B_j=\partial_\ell$ is a signed adjacency matrix of the bipartite incidence graph between the $(\ell-1)$-cells and the $\ell$-cells. 
The conditions above make this graph simple, $C_4$-free and $(d,k+1)$-biregular, so it fits the setting of Section~\ref{sec:det_intro}. If $(\Delta_j)$ is a simplicial complex, then with the choice of $\ell=k$, the conditions \eqref{polytop1}, \eqref{polytop2} and \eqref{polytop4} are always satisfied.


In this general setting, there is a corresponding natural $1$-out complex $\mathcal{T}_j^{1o}$: it contains the $(\ell-1)$-skeleton of $\Delta_j$, and each $(\ell-1)$-cell independently chooses a uniform $\ell$-cell containing it, which is then included in $\mathcal{T}_j^{1o}$.  
Theorem~\ref{thmgeneralcoupling} gives a coupling of the determinantal hyperforest $\mathcal{T}_j$ and $\mathcal{T}_j^{1o}$ 
under which they differ by $o(f_{\ell-1}(\Delta_j))$ top-dimensional cells in expectation.

A prominent special case is given by taking $\Delta$ to be the complete $k$-dimensional simplicial complex on the vertex set $[N]$. 
Here $\ell=k$, $f_{k-1}=\binom{N}{k}$ and $d=N-k$. 
For $k\geq2$, a spanning hyperforest in this complex is called a \emph{$k$-hypertree}: it has a complete $(k-1)$-skeleton, precisely $\binom{N-1}{k}$ faces of dimension $k$, and a finite homology group $H_{k-1}(S,\mathbb{Z})$. 
Kalai's generalization of Cayley's formula~\cite{kalai1983enumeration} states that
\begin{equation}\label{kalaiformula}
    \sum_S |H_{k-1}(S,\mathbb{Z})|^2=N^{\binom{N-2}{k}}\,,
\end{equation}
where the sum is over all $k$-hypertrees on $[N]$. The corresponding \emph{determinantal hypertree} is denoted by $\mathcal{T}_{N,k}$, and 
it has been studied extensively~\cite{kahle2022topology,meszaros2022local,werf2022determinantal,meszaros2023coboundary,meszaros20252,meszaros2025bounds,meszaros2025using,meszaros2026cocycles}.

Additionally, the first author~\cite{meszaros2026homology} recently proved that the homology torsion of the determinantal hypertree has a deterministic exponential growth rate. More precisely, as $N\to\infty$,
\[
    \frac{\log|H_{k-1}(\mathcal{T}_{N,k},\mathbb{Z})|}{\binom{N}{k}}
    \text{ converges to }c_k\text{ in probability}.
\]
Since $c_k>0$ for $k\geq2$, the homology group typically has exponentially large order. 

As a consequence of Theorem~\ref{thm1}, we extend this result to determinantal hyperforests in the polytopal setting above under an assumption on the weighted enumeration $\tau_\ell(\Delta_j)$,
which is the higher dimensional weighted analogue of the conclusion of Corollary~\ref{CorNumberOfTrees}.

\begin{theorem}\label{thmBTB2}
    In the polytopal setup above, let $\mathcal{T}_j$ be a determinantal $\ell$-hyperforest of $\Delta_j$. Assume that, as $j\to\infty$,
    \begin{equation}\label{nicetau}
\frac{\log(\tau_\ell(\Delta_j))}{f_{\ell-1}(\Delta_j)}-\log d\to 0\,.
    \end{equation}
    Then, as $j\to\infty$,
    \[
        \frac{\log|\tors H_{\ell-1}(\mathcal{T}_j,\mathbb{Z})|}{f_{\ell-1}(\Delta_j)}
        \text{ converges to }c_k\text{ in probability.}
    \]
\end{theorem}

The assumption~\eqref{nicetau} holds for the complete simplicial complex by Kalai's formula, so Theorem~\ref{thmBTB2} recovers the result of~\cite{meszaros2026homology}. We prove the theorem in Section~\ref{SecTop}, 
where we also give spectral sufficient conditions for~\eqref{nicetau} for complexes that are acyclic in positive codimension.

The remainder of the paper is organized as follows. In Section~\ref{SecPreliminaries}, we recall the preliminaries needed for the proofs. In Section~\ref{SecUSTCoupling}, we prove Theorems~\ref{thm1dcoupling} and~\ref{thm:local-coupling}. Section~\ref{proofofthm1} is devoted to the proof of the asymptotic equipartition property. In Section~\ref{SecGeneralCoupling}, we prove Theorem~\ref{thmgeneralcoupling}, and in Section~\ref{SecAEPCorollaries} we derive further consequences of the asymptotic equipartition property. Finally, in Section~\ref{SecTop}, we prove Theorem~\ref{thmBTB2} and discuss spectral sufficient conditions for homology torsion growth.

\section{Preliminaries}\label{SecPreliminaries}
For the proof of Theorem~\ref{thm1dcoupling} only Section~\ref{SecMarton} is needed from this section.

\subsection{Marton's coupling lemma}\label{SecMarton}

Given two random variables $X$ and $Y$ taking values in a finite set $\mathcal{X}$, their Kullback–Leibler divergence is defined as
\[\dkl(Y||X)=\sum_{x\in \mathcal{X}}\mathbb{P}(Y=x)\log\frac{\mathbb{P}(Y=x)}{\mathbb{P}(X=x)}.\]
Note that if $\mathbb{P}(Y=x)=0$, then the corresponding term is interpreted as $0$.

\begin{lemma}[Marton \cite{marton1986simple}]\label{lemmaMarton}
    Let $X=(X(1),X(2),\dots,X(n))$ and $Y=(Y(1),Y(2),\dots,Y(n))$ be two sequences of random variables taking values in some finite set. Assume that $X(1),X(2),\dots,X(n)$ are independent. Then there is a coupling of $X$ and $Y$ such that under this coupling
    \[\frac{1}n\sum_{i=1}^n \mathbb{P}(X(i)\neq Y(i))\le \sqrt{\frac{1}n \dkl(Y||X)}.\]
\end{lemma}

\subsection{The space of rooted graphs}\label{SecRootedGraphs}
A rooted graph is a pair $(G,o)$, where $G$ is a locally finite connected graph, and $o$ is a distinguished vertex of $G$ called the root. Given a rooted graph $(G,o)$ and a radius $r$, the $r$-neighborhood $B_r(G,o)$ of the root is defined as the subgraph of $G$ induced by the vertices which are at most distance $r$ from the root. Choosing $o$ as the root, we can regard $B_r(G,o)$ as a rooted graph. Let $\mathcal{G}_\bullet$ be the set of rooted graphs up to root preserving isomorphisms. Then $\mathcal{G}_\bullet$ can be turned into a metric space as follows. Given two rooted graphs $(G_1,o_1)$ and $(G_2,o_2)$, let $r$ be the largest radius, possibly $\infty$, such that there is a root preserving isomorphism between $B_r(G_1,o_1)$ and $B_r(G_2,o_2)$. Then the distance of $(G_1,o_1)$ and $(G_2,o_2)$ is defined to be $2^{-r}$, with $2^{-\infty}=0$. This metric structure on $\mathcal{G}_\bullet$ also induces a measurable structure, so it makes sense to speak about random rooted graphs, and the weak convergence of random rooted graphs. This convergence notion was first introduced by Benjamini and Schramm~\cite{benjamini2001recurrence}, see also~\cite{aldous2007processes}. 

A signed graph $(G,b)$ is a graph $G$ together with a labeling of the edges $b:E(G)\to\{-1,+1\}$.

A signed rooted graph is a triple $(G,o,b)$, where $(G,o)$ is a rooted graph and $b:E(G)\to \{-1,+1\}$ is a labeling of the edges. Similarly as above, let $\mathcal{G}^{\pm}_\bullet$ be the set of signed rooted graphs up to isomorphisms. The same way as before, $\mathcal{G}^{\pm}_\bullet$ can be equipped with  a metric, the only difference is that isomorphisms are required to preserve both the root and the edge labels.

Any finite graph $G$ can be turned into a random rooted graph by considering $(G_o,o)$, where $o$ is a uniform random vertex of $G$ and $G_o$ is the connected component of $o$. To ease the notation, with a slight abuse we will often just write $(G,o)$ in place of $(G_o,o)$. Also, when $G$ is bipartite, we often want to choose the random root as a uniform random vertex from one specific color class.

\subsection{Spectral measures}\label{SecSpectralMeasure}

Consider a signed graph $(G,b)$ such that $V(G)$ is countable and $G$ is locally finite. Given a vertex $v\in V(G)$, let $\chi_v\in \ell^2(V(G))$ be the characteristic vector of $v$. The adjacency operator $A=A_{(G,b)}$ of $(G,b)$ is a symmetric operator densely defined on the set of finitely supported functions of $\ell^2(V(G))$ such that
\[\langle A \chi_v,\chi_u\rangle=\begin{cases}
b(uv)&\text{ if }uv\in E(G),\\
0&\text{ otherwise.}
\end{cases}\]

Since $A$ is symmetric, it has a closure, which we also denote by $A$. Note that $A$ is not necessarily self-adjoint. However, we will see that in all the cases relevant to us $A$ will be self-adjoint.

Let $(G,o,b)$ be a signed rooted graph such that $A=A_{(G,b)}$ is self-adjoint. Then by the spectral theorem, $f(A)$ is a bounded operator for every bounded Borel function $f:\mathbb{R}\to\mathbb{R}$. Moreover, there is a probability measure $\nu=\nu_{(G,o,b)}$ on~$\mathbb{R}$ such that
\[\langle f(A)\chi_o,\chi_o\rangle=\int f(t) \,\nu(dt)\]
for all bounded Borel functions $f:\mathbb{R}\to\mathbb{R}$. We call $\nu$ the spectral measure of $(G,o,b)$. 

Given a graph $G$, we can always consider it as a signed graph by setting $b(e)=1$ for all the edges of~$G$. Thus, it makes sense to speak about the adjacency operator of locally finite graphs on a countable vertex set, and we can also define the spectral measure $\nu_{(G,o)}$ of a rooted graph.

For the proof of the next lemma, see \cite[Section 3.2]{bordenave2011rank}.

\begin{lemma}\label{Lemmaspectralconvergence}
Let $(G_n,o_n,b_n)$ be a sequence of random signed rooted graphs such that they converge weakly to the random signed rooted graph $(G,o,b)$. Assume that the adjacency operators of $(G_n,b_n)$ and $(G,b)$ are self-adjoint almost surely. Then $\mathbb{E}\nu_{(G_n,o_n,b_n)}$ converges weakly to $\mathbb{E}\nu_{(G,o,b)}$.

\end{lemma}

In certain cases the signs assigned to the edges can be ignored.

\begin{lemma}\label{treelimit}
Let $(G_n,o_n,b_n)$ be a sequence of random signed rooted graphs. Assume that the underlying random rooted graphs $(G_n,o_n)$ converge weakly to a random rooted graph $(G,o)$ such that $G$ is a tree almost surely. Assume that the adjacency operators of $(G_n,b_n)$ and $G$ are self-adjoint almost surely. Then $\mathbb{E}\nu_{(G_n,o_n,b_n)}$ converges weakly to $\mathbb{E}\nu_{(G,o)}$.
\end{lemma}
\begin{proof}
    Let $(G,o)$ be a rooted graph and let $b,b':E(G)\to\{-1,+1\}$ be two assignments of signs. We say that $b$ and $b'$ are similar if there is a map $\Psi:V(G)\to \{-1,+1\}$ such that 
    $b'(uv)=\Psi(u)\Psi(v)b(uv)$ for all $uv\in E(G)$. Let $U$ be the bounded linear operator on $\ell^2(V(G))$ for which
    $U\chi_v=\Psi(v)\chi_v$ for all $v\in V(G)$. Note that $U$ is unitary, and $A_{(G,b')}=U^*A_{(G,b)}U$. Thus, for any bounded Borel function $f:\mathbb{R}\to\mathbb{R}$, we have
    \begin{multline*}
        \langle f(A_{(G,b')})\chi_o,\chi_o\rangle=\langle U^*f(A_{(G,b)})U\chi_o,\chi_o\rangle=\langle f(A_{(G,b)})U\chi_o,U\chi_o\rangle\\=\langle f(A_{(G,b)})\Psi(o)\chi_o,\Psi(o)\chi_o\rangle=\langle f(A_{(G,b)})\chi_o,\chi_o\rangle=\int f(t) \,\nu_{(G,o,b)}(dt).
        \end{multline*}
    This shows that $(G,o,b)$ and $(G,o,b')$ have the same spectral measure.

    Next, we replace $b_n$ with a similar assignment of signs $b_n'$ as follows. Let $r$ be the largest (random) radius such that $B_r(G_n,o_n)$ is a tree, allowing $r=\infty$ and taking $B_\infty(G_n,o_n)$ to be the component of $o_n$. Since $B_r(G_n,o_n)$ is a tree it is easy to see that one can find an assignment of signs $b_n'$ similar to $b_n$ such that $b_n'(e)=1$ for each edge $e$ of $B_r(G_n,o_n)$. Since the limiting rooted graph $(G,o)$ is a tree, it follows that $(G_n,o_n,b_n')$ converges to $(G,o,\mathbbm{1})$, where $\mathbbm{1}$ is the constant $1$ assignment. Then the statement follows from Lemma~\ref{Lemmaspectralconvergence}. 
\end{proof}

Given an $n\times n$ symmetric matrix $M$ with eigenvalues $\lambda_1,\lambda_2,\dots,\lambda_n$, the empirical spectral measure of $M$ is defined as
\[\eta_M=\frac{1}n\sum_{i=1}^n \delta_{\lambda_i},\]
where $\delta_{\lambda_i}$ is the Dirac measure on $\lambda_i$.

Now let $(G,b)$ be a finite bipartite signed graph with color classes $V$ and $U$. The biadjacency matrix $B$ of $(G,b)$ is a matrix where the rows are indexed by $V$ and the columns are indexed by $U$, and
\[B_{vu}=\begin{cases} b(uv)&\text{ if $uv$ is an edge,}\\0&\text{ otherwise.}\end{cases}\]

\begin{lemma}\label{emiprical}
    Let $o$ be a uniform random vertex in $V$. Then the empirical spectral measure $\eta$ of $BB^T$ can be obtained as the push forward of the measure $\mathbb{E}\nu_{(G,o,b)}$ by the map $t\mapsto t^2$. 
\end{lemma}
\begin{proof}
    The adjacency matrix (operator) of $(G,b)$ is of the form
    \[A=\begin{pmatrix}
        0 &B\\
        B^T&0
    \end{pmatrix},\]
    where the first rows are indexed by $V$ and the last rows are indexed by $U$. Then
    \[A^2=\begin{pmatrix}
        BB^T&0\\
        0&B^TB
    \end{pmatrix}.\]
    Let $\eta'$ be the pushforward of $\mathbb{E}\nu_{(G,o,b)}$ by the map $t\mapsto t^2$. Let $f:\mathbb{R}\to\mathbb{R}$ be any bounded Borel function. Then
    \[\int f(t) \eta(dt)=\frac{1}{|V|}\Tr(f(BB^T))=\mathbb{E}\langle f(A^2)\chi_o,\chi_o\rangle=\mathbb{E}\int f(t^2)\,\nu_{(G,o,b)}(dt)=\int f(t)\,\eta'(dt).\]
    Thus, the statement follows.
\end{proof}

\subsection{The local weak limit of determinantal processes}\label{seclocallimit}

Given a subset $\mathcal{T}_n$ of $U_n$, let $\mathfrak{G}_n[\mathcal{T}_n]$ be the subgraph of $\mathfrak{G}_n$ induced by $V_n\cup\mathcal{T}_n$. 

Let $\mathcal{T}_n$ be drawn from the measure $\mathbb{P}^{H_n}$. The results of Nachmias and the second author~\cite{nachmias2025local} identified the quenched local weak limit of $\mathfrak{G}_n[\mathcal{T}_n]$. Since this is a  quenched convergence, it requires considering probability measures on probability measures, which, although it can be done, might be confusing. Thus, here, relying on Skorokhod's representation theorem, we rephrase the quenched convergence in terms of a coupling. See also~\cite[Proof of Lemma 13]{meszaros2026homology} for how to get the coupling formulation from the quenched convergence. 

\begin{theorem}\label{thmNachmiasPeled}
    There is a coupling of $(\mathcal{T}_n)_{n\geq1}$, with $\mathcal{T}_n\sim\mathbb{P}^{H_n}$, such that the following holds almost surely. For the realized sequence $(\mathcal{T}_n)$, the distribution of $(\mathfrak{G}_n[\mathcal{T}_n],o_n)$, where $o_n$ is chosen uniformly from $V_n$, converges weakly to the law of a random rooted tree $(\mathbb{T}_k,o)$. The limiting law depends only on $k$. 
\end{theorem}

\begin{lemma}\label{almostsurleyselfadjoint}
The adjacency operator of $\mathbb{T}_k$ is almost surely self-adjoint.
\end{lemma}
\begin{proof}
    The proof of \cite[Claim 3.3.]{linial2016phase} can be slightly modified to give this statement.
\end{proof}

\begin{lemma}\label{NachmiasPeledSpectral}
Let $\eta_{\mathcal{T}_n}$ be the empirical spectral measure of $B_n[\mathcal{T}_n](B_n[\mathcal{T}_n])^T$, and let $\eta_k$ be the pushforward of $\mathbb{E}\nu_{(\mathbb{T}_k,o)}$ by the map $t\mapsto t^2$. Under the coupling provided by Theorem~\ref{thmNachmiasPeled}, the measures $\eta_{\mathcal{T}_n}$ converge weakly to~$\eta_k$ almost surely.   
\end{lemma}
\begin{proof}
    This lemma follows by combining Lemma~\ref{treelimit}, Lemma~\ref{emiprical} and Lemma~\ref{almostsurleyselfadjoint}. 
\end{proof}

Finally, we can define the constant $c_k$ that appears in Theorem~\ref{thm1}. 

\begin{definition}\label{ckdef}
    Let $\eta_k$ be the pushforward of $\mathbb{E}\nu_{(\mathbb{T}_k,o)}$ by the map $t\mapsto t^2$, and let
    \[c_k=\frac{1}2\int \log(t) \eta_k(dt).\]
\end{definition}

We included the factor $\frac{1}2$ in the definition of $c_k$ to be consistent with~\cite{meszaros2026homology}.
\begin{lemma}\label{etak0}
    We have $\eta_k(\{0\})=0$.
\end{lemma}
\begin{proof}
    See \cite[Lemma 17]{meszaros2026homology} or \cite[Lemma 4.2]{linial2019enumeration}.
\end{proof}

Note that it was proved in \cite{meszaros2026homology} that

\[\frac{1}2\log\left(\frac{k+1}e\right)\le c_k\le \frac{1}2\log\left({k+1}\right).\]
\subsection{The matching measure}

Let $G$ be a finite graph. The matching polynomial of $G$ is defined as
\[\mu(G,x)=\sum_k (-1)^k m_k(G) x^{|V(G)|-2k},\]
where $m_k(G)$ is the number of matchings in $G$ with exactly $k$ edges. This is a real rooted polynomial as Heilmann and Lieb proved~\cite{heilmann1972theory}. The matching measure $\rho_G$ of $G$ is defined as the uniform measure on the roots of $\mu(G,x)$, counted with multiplicity. This measure was introduced by Ab\'ert, Csikv\'ari, Frenkel and Kun~\cite{abert2016matchings}  as a
tool to understand matchings of locally convergent graph sequences, see also~\cite{abert2015matching,csikvari2016matchings,csikvari2017lower,bencs2022atoms}.

Let $(G,o)$ be a rooted graph. Let $T_o(G)$ be the path tree of $o$, that is, the vertices of $T_o(G)$ are the (self-avoiding) paths starting from $o$, and two vertices of $T_o(G)$ are connected if one of the corresponding paths can be obtained from the other by extending it with a vertex. Let $\bar{o}$ be the path consisting of the single vertex $o$. Then ${(T_o(G),\bar{o})}$ is a rooted tree. Let $\nu_{(T_o(G),\bar{o})}$ be the spectral measure of $(T_o(G),\bar{o})$ as defined in Section~\ref{SecSpectralMeasure}.

Then the matching measure can be disintegrated as
\[\rho_G=\mathbb{E}\nu_{(T_{o}(G),\bar{o})},\]
where $o$ is a uniform random vertex of $G$, see \cite[Theorem 2.2.]{abert2015matching}.

When $G$ is bipartite with color classes $V$ and $U$, we would like a disintegration that only uses vertices from~$V$.

\begin{lemma}\label{lemmadisintegrate}
    Assume that $|V|\ge |U|$. Then
    \[\rho_G=\frac{2|V|}{|U|+|V|}\left(\mathbb{E} \nu_{(T_o(G),\bar{o})}-\frac{|V|-|U|}{2|V|}\delta_0\right)\]
    where $o$ is a uniform random vertex of $V$.
\end{lemma}
\begin{proof}
    First, we consider the case when $|V|=|U|=h$. In this case we need to prove that
    \[\rho_G=\mathbb{E} \nu_{(T_o(G),\bar{o})}.\]

    We prove this by showing that the two sides have the same Stieltjes-transform. Recall that given a probability measure $\nu$ on $\mathbb{R}$, its  Stieltjes-transform is defined as
    \[S_{\nu}(z)=\int\frac{1}{z-t} \nu(dt)\text{ for }z\in \mathbb{C}\setminus\mathbb{R}.\]

    It is straightforward to see that
    \[S_{\rho_G}(z)=\frac{1}{2h}\frac{\mu'(G,z)}{\mu(G,z)}.\]

    The coefficient of $z^{2h-2k-1}$ in the polynomial $\mu'(G,z)$ is $(-1)^k(2h-2k)m_k(G)$. Observe that $(h-k)m_k(G)$ is exactly the number of pairs $(M,v)$, where $M$ is a matching of size $k$ and $v$ is a vertex in $V$ not covered by $M$. Thus, if $m_k(G,v)$ is the number of matchings of size $k$ not covering the vertex $v$, then $(h-k)m_k(G)=\sum_{v\in V} m_k(G,v)$. Note that $(-1)^k\sum_{v\in V} m_k(G,v)$ is the coefficient of $z^{2h-2k-1}$ in the polynomial $\sum_{v\in V} \mu(G-v,z)$. Thus, 
     
     \[\mu'(G,z)=2\sum_{v\in V} \mu(G-v,z).\]
     Therefore,
     \[S_{\rho_G}(z)=\frac{1}{2h}\frac{\mu'(G,z)}{\mu(G,z)}=\frac{1}h\sum_{v\in V}\frac{\mu(G-v,z)}{\mu(G,z)}.\]

By \cite[Lemma 2.7.]{bencs2022atoms}, we have $S_{\nu_{(T_o(G),\bar{o})}}(z)=\frac{\mu(G-o,z)}{\mu(G,z)}$. Thus, the statement follows in the special case of $|V|=|U|$. 

Now, consider the case when $|V|>|U|$. Let $G'$ be the bipartite graph obtained from $G$ by adding $|V|-|U|$ isolated vertices to $U$. Using the already established case, we have
\[\rho_{G'}=\mathbb{E} \nu_{(T_o(G'),\bar{o})}=\mathbb{E} \nu_{(T_o(G),\bar{o})}.\]
Observing that $\mu(G',z)=z^{|V|-|U|}\mu(G,z)$, the statement follows. 
\end{proof}

\begin{lemma}\label{mathingmeasureconvergence}
    Let $G_n$ be a sequence of finite bipartite graphs with color classes $V_n$ and $U_n$ such that \break $|V_n|\ge |U_n|\ge (1-o(1))|V_n|$. Let $o_n$ be a uniform random vertex of $V_n$. Assume that $(G_n,o_n)$ converges locally to a random rooted tree $(\mathbb{T},o)$, where the adjacency operator of $\mathbb{T}$ is almost surely self-adjoint. Then $\rho_{G_n}$ converges weakly to the expected  spectral measure of $(\mathbb{T},o)$.\end{lemma}
    \begin{proof}
    Note that given a rooted graph $(G,o)$, we can tell $B_r(T_o(G),\bar{o})$ just based on $B_r(G,o)$. Therefore, $(T_{o_n}(G_n),\bar{o}_n)$ converges weakly to $(T_o(\mathbb{T}),\bar{o})$. Since $\mathbb{T}$ is a tree, we have $(T_o(\mathbb{T}),\bar{o})\cong (\mathbb{T},o)$. Therefore by Lemma~\ref{Lemmaspectralconvergence}, we see that $\mathbb{E}\nu_{(T_{o_n}(G_n),\bar{o}_n)}$\ converges to $\mathbb{E}\nu_{(\mathbb{T},o)}$. By Lemma~\ref{lemmadisintegrate}, it follows that $\rho_{G_n}$ also has the same limit.
\end{proof}

\section{Coupling of uniform spanning trees and $1$-out graphs}\label{SecUSTCoupling}

\subsection{Constructing the coupling -- The proof of Theorem~\ref{thm1dcoupling}}\label{section1dcoupling}

Let $V=V(G)$, and let us fix a vertex $v_0\in V$. Let $\mathcal{F}$ be the set of functions $f:V\setminus \{v_0\}\to V$ such that $f(v)$ is a neighbor of $v$ (in the graph $G$) for all $v\in V\setminus\{v_0\}$. Obviously, $|\mathcal{F}|=\prod_{v\in V\setminus\{v_0\}} \deg(v)$. 

Given $f\in \mathcal{F}$, let $G_f$ be the subgraph of $G$ consisting of the edges $\{vf(v)\,:\,v\in V\setminus \{v_0\}\}$. If $T$ is a spanning tree of $G$, then there is a unique $f\in \mathcal{F}$ such that $G_f=T$. Indeed, let us orient the edges of $T$ towards the vertex $v_0$. Then each vertex other than $v_0$ has out-degree $1$. For a vertex $v\neq v_0$, let us define $f(v)$ as the unique neighbor of $v$ such that the edge $vf(v)$ is an edge of $T$ oriented away from~$v$. It is straightforward to see that $G_f=T$, and no other element of $\mathcal{F}$ has the same property. Let $\mathcal{F}'=\{f\in \mathcal{F}\,:\,G_f\text{ is a spanning tree}\}$. By the discussion above $|\mathcal{F}'|$ is just the number of spanning trees of $G$. Thus, by a result of Kostochka~\cite{kostochka1995number}, we have
\[|\mathcal{F}'|\ge \exp\left(-nO\left(\frac{\log^2 k}{k}\right)\right)\prod_{v\in V}\deg(v).\]

Let $X$ and $Y$ be uniform random elements of $\mathcal{F}$ and $\mathcal{F}'$, respectively. Then

\[\dkl(Y||X)=\sum_{f\in \mathcal{F}'}\mathbb{P}(Y=f)\log\frac{\mathbb{P}(Y=f)}{\mathbb{P}(X=f)}=\log\frac{|\mathcal{F}|}{|\mathcal{F}'|}\le nO\left(\frac{\log^2 k}{k}\right).\]

Note that the random variables $(X(v))_{v\in V\setminus\{v_0\}}$ are independent, thus, Lemma~\ref{lemmaMarton} can be applied to give a coupling of $X$ and $Y$ such that 

\[\frac{1}{n-1}\sum_{v\in V\setminus\{v_0\}}\mathbb{P}(X(v)\neq Y(v))\le \sqrt{\frac{1}{n-1}\dkl(Y||X)}\le O\left(\frac{\log k}{\sqrt{k}}\right).\]

Let $e$ be a uniform random edge incident to $v_0$ independent of $(X,Y)$. Then $(G_X\cup\{e\},G_Y)$ gives a coupling of $G^{1o}$ and $T$. Under this coupling
\[\mathbb{E}\frac{|E(G^{1o})\Delta E(T)|}n\le \frac{1}n\left(1+2\sum_{v\in V\setminus\{v_0\}}\mathbb{P}(X(v)\neq Y(v))\right)=O\left(\frac{\log k}{\sqrt{k}}\right).\]

\subsection{Coupling rooted neighborhoods -- Proof of Theorem~\ref{thm:local-coupling}}

Fix a radius $r\geq 0$. For each $v\in V(G)$, write $N(v)$ for its set of neighbors in $G$. Throughout this proof, use the coupling from Theorem~\ref{thm1dcoupling} and choose a uniform vertex $o$ independently of the pair $(G^{1o},T)$. Then we have 
\begin{align*}1+\deg_{G^{1o}}(v)&=1+\mathbbm{1}(R(R(v))\neq v)+\sum_{u\in N(v)}\mathbbm{1}(R(u)=v)\\&\le 2+\sum_{u\in N(v)}\mathbbm{1}(R(u)=v).\end{align*}
Therefore,
\begin{align*}
    \mathbb{E}\left(1+\deg_{G^{1o}}(v)\right)^2&\le 4+5\sum_{u\in N(v)}\mathbb{P}(R(u)=v)+\sum_{\substack{u_1,u_2\in N(v)\\u_1\neq u_2}} \mathbb{P}(R(u_1)=v\text{ and }R(u_2)=v)\\&\le
    4+5\frac{\deg_G(v)}k+\frac{\deg_G(v)(\deg_G(v)-1)}{k^2}\\&\le 10 \left(\frac{\deg_G(v)}{k}\right)^2.
\end{align*}

Therefore, for this uniform root $o$, we have
\[\mathbb{E}\left(1+\deg_{G^{1o}}(o)\right)^2\le 10 \mathbb{E}\left(\frac{\deg_G(o)}{k}\right)^2=10\varrho.\]

For a subset $U$ of $V(G)$ and a subgraph $H$ of $G$, let $B_{H,r}(U)$ be the $r$-neighborhood of $U$ in the graph $H$, that is, the set of all the vertices which are at distance at most $r$ from some vertex of $U$ in the graph $H$.

For a random subset $U$ such that $o$ is independent of the pair $(G^{1o},U)$, the Cauchy--Schwarz inequality gives
\[\mathbb{P}(o\in B_{G^{1o},1}(U))\le \mathbb{E}\mathbbm{1}(o\in U)(1+\deg_{G^{1o}}(o))\le \sqrt{\mathbb{P}(o\in U)\mathbb{E}\left(1+\deg_{G^{1o}}(o)\right)^2}\le \sqrt{10\varrho}\sqrt{\mathbb{P}(o\in U)}.\]

Observing that $B_{G^{1o},i}(U)=B_{G^{1o},1}(B_{G^{1o},i-1}(U))$, it can be proved by induction that
\[\mathbb{P}(o\in B_{G^{1o},r}(U))\le \left(\mathbb{P}(o\in U)\right)^{2^{-r}}(\sqrt{10\varrho})^{\sum_{i=0}^{r-1} 2^{-i} }.\]
Thus,
\[\mathbb{P}(o\in B_{G^{1o},r}(U))\le 10\varrho\left(\mathbb{P}(o\in U)\right)^{2^{-r}}.\]

Let $U$ be the set of vertices that are incident to an edge in $E(G^{1o})\Delta E(T)$. Then
\[\mathbb{P}(o\in U)=O\left(\frac{\log k}{\sqrt{k}}\right).\]

Note that if $o\notin B_{G^{1o},r}(U)$, then $B_T(o,r)= B_{G^{1o}}(o,r)$. Thus,
\[ 
        \mathbb{P}\left(B_T(o,r)\neq B_{G^{1o}}(o,r)\right)  \le 10\varrho\left(\mathbb{P}(o\in U)\right)^{2^{-r}}\le 10\varrho \left(O\left(\frac{\log k}{\sqrt{k}}\right)\right)^{2^{-r}}.\]

\section{
The asymptotic equipartition property -- Proof of Theorem~\ref{thm1}}\label{proofofthm1}

\subsection{Proof outline}

Given a subset $\mathcal{T}$ of $U$, let $B[\mathcal{T}]$ be the submatrix of $B$ consisting of the columns indexed by the elements of $\mathcal{T}$. If $M$ is a matrix, where the rows and columns are both indexed by $U$, let $M\restriction \mathcal{T}$ be the principal submatrix of $M$ determined by the rows and columns indexed by the elements of $\mathcal{T}$. 

Moreover,
\begin{itemize}
\item let $\psi_{\mathcal{T}}$ be the empirical spectral measure of $d P_H\restriction \mathcal{T}$, 
\item let $\eta_{\mathcal{T}}$ be the empirical spectral measure of $B[\mathcal{T}](B[\mathcal{T}])^T$, and 
\item let $\eta_{\mathcal{T}}'$ be the empirical spectral measure of $(B[\mathcal{T}])^TB[\mathcal{T}]=B^TB\restriction \mathcal{T}$.
\end{itemize}

Let $\eta_k$ be the measure defined in Definition~\ref{ckdef}. Let  $\mathcal{T}_1,\mathcal{T}_2,\dots$ be the coupling provided by Theorem~\ref{thmNachmiasPeled}. We can express $\mathbb{P}^{H_n}(\mathcal{T}_n)$ in terms of the measure $\psi_{\mathcal{T}_n}$ by the formula
\begin{equation}\label{Pasintlog0}\frac{1}r\log\mathbb{P}^{H_n}(\mathcal{T}_n)=-\log(d)+\int \log(t)\,\psi_{\mathcal{T}_n}(dt).\end{equation}
In Lemma~\ref{psietaconv}, we show that $\psi_{\mathcal{T}_n}$ converges weakly to $\eta_k$ almost surely. The starting point of the proof is the result of Nachmias and the second author~\cite{nachmias2025local} on the local weak limit of $\mathcal{T}_n$, which we recalled in Theorem~\ref{thmNachmiasPeled}. As we explained in Section~\ref{seclocallimit}, this local weak convergence implies that $\eta_{\mathcal{T}_n}$ converges to $\eta_k$. The key idea of the proof of Lemma~\ref{psietaconv} is to compare $\psi_{\mathcal{T}_n}$ and $\eta_{\mathcal{T}_n}$, and prove that they have the same weak limit. Once we have proved this, it follows that $\psi_{\mathcal{T}_n}$ also converges to $\eta_k$. By definition, this weak convergence means that for every bounded continuous function $f$, we have
\begin{equation}\label{intlimitchange}\lim_{n\to\infty} \int f(t)\,\psi_{\mathcal{T}_n}(dt)=\int f(t)\,\eta_k(dt).
\end{equation}
By~\eqref{Pasintlog0} and the definition of $c_k$ in Definition~\ref{ckdef}, it suffices to prove that equation~\eqref{intlimitchange} also holds for $f(t)=\mathbbm{1}(t>0)\log(t)$. The change from normalization by $r$ to normalization by $n$ is justified by~\eqref{ranklognormalization}, as we explain at the end of the proof. However, $\mathbbm{1}(t>0)\log(t)$ is not bounded. Thus, besides the weak convergence of $\psi_{\mathcal{T}_n}$, further conditions are required to ensure that \eqref{intlimitchange} holds for $f(t)=\mathbbm{1}(t>0)\log(t)$. Lemma~\ref{logdetgeneral} gives such sufficient conditions. One of the additional conditions is a uniform integrability condition; we prove that this condition holds for $\psi_{\mathcal{T}_n}$ in Lemma~\ref{lemmauniformint}.

\subsection{Details of the proof of Theorem~\ref{thm1}}

\begin{lemma}\label{psietaconv} Let  $\mathcal{T}_1,\mathcal{T}_2,\dots$ be the coupling provided by Theorem~\ref{thmNachmiasPeled}. Then with probability one, $\eta_{\mathcal{T}_n}$, $\eta_{\mathcal{T}_n}'$ and $\psi_{\mathcal{T}_n}$ all converge weakly to $\eta_k$.
\end{lemma}
\begin{proof}

The convergence of $\eta_{\mathcal{T}_n}$ to $\eta_k$ was proved in Lemma~\ref{NachmiasPeledSpectral}. Note that $B_n[\mathcal{T}_n](B_n[\mathcal{T}_n])^T$ and $(B_n[\mathcal{T}_n])^TB_n[\mathcal{T}_n]$ have the same nonzero eigenvalues (with the same multiplicities). Observe that \break $B_n[\mathcal{T}_n](B_n[\mathcal{T}_n])^T$ is an $n\times n$ matrix and $B_n^TB_n\restriction \mathcal{T}_n$ is an $r_n\times r_n$ matrix. As we will see in \eqref{rvsn}, $n-r_n=o(n)$. Thus, it is straightforward to see that with probability $1$, the empirical spectral measure of $B_n^TB_n\restriction \mathcal{T}_n$, that is, $\eta'_{\mathcal{T}_n}$ also converges to $\eta_k$.

Now we move on to the proof of the convergence of $\psi_{\mathcal{T}_n}$.  
Let $\lambda_1,\lambda_2,\dots, \lambda_r$ be the nonzero eigenvalues of $\frac{1}dB^TB$, and let $g_1,g_2,\dots, g_r$ be the corresponding orthogonal unit eigenvectors. Note that the vectors $g_1,g_2,\dots, g_r$ form an orthonormal basis of $H$. Then
\begin{align*}
\frac{1}dB^TB&=\sum_{i=1}^r \lambda_i g_ig_i^T,\text{ and }
\\P_H&=\sum_{i=1}^r g_ig_i^T.
\end{align*}

Choose $\varepsilon=\varepsilon_n\in (0,1)$, and let 
\begin{equation}\label{Idef}I=\{i\in [r]\,:\,(\lambda_i-1)^2\le \varepsilon\}.\end{equation}

By \cite[Claim 2.3.]{nachmias2025local}, we have
\begin{equation}\label{Ibound}(1-k(\varepsilon d)^{-1})n\le |I|\le n.\end{equation}

In the rest of the proof, we will choose $\varepsilon$ such that $\varepsilon\to 0$ and $\varepsilon d\to\infty$. (Since $d\to\infty$ such a choice exists.)

Note that $|I|\le r$, so \eqref{Ibound} also implies that
\begin{equation}\label{rvsn}
    r=n-o(n).
\end{equation}

Applying~\eqref{Ibound} with $\varepsilon=1/2$ gives the stronger estimate
\begin{equation}\label{ranklognormalization}
    0\leq n-r\leq\frac{2kn}{d},
    \qquad
    \left(1-\frac{r}{n}\right)\log d\longrightarrow0\,.
\end{equation}

Given a subset $\mathcal{T}$ of $U$, let $g_{\mathcal{T},i}\in \mathbb{R}^\mathcal{T}$ be the projection of the vector $g_i$ onto the components indexed by~$\mathcal{T}$. Observe that
\begin{align*}
B^TB\restriction \mathcal{T}&=\sum_{i=1}^r d\lambda_{i} g_{\mathcal{T},i}g_{\mathcal{T},i}^T,\text{ and }
\\P_H\restriction\mathcal{T}&=\sum_{i=1}^r g_{\mathcal{T},i}g_{\mathcal{T},i}^T.
\end{align*}

Note that $g_{\mathcal{T},1},\dots,g_{\mathcal{T},r}$ is not an orthonormal system, so $\lambda_1,\dots,\lambda_{r}$ are not the eigenvalues of the matrix $\frac{1}d B^TB\restriction \mathcal{T}$.

We have seen that the empirical spectral measure of $B^TB\restriction \mathcal{T}_n=\sum_{i=1}^r (d\lambda_{i}) g_{\mathcal{T}_n,i}g_{\mathcal{T}_n,i}^T$ converges to~$\eta_k$. Combining this with the fact that $|[r]\setminus I|=o(r)$ and Lemma~\ref{smallrankupdate} below, we see that with probability~$1$,
\begin{equation}\label{eqstep1}
  \text{the empirical spectral measure of }\sum_{i\in I}(d\lambda_{i}) g_{\mathcal{T}_n,i}g_{\mathcal{T}_n,i}^T\text{ converges weakly to }\eta_k. 
\end{equation}

\begin{lemma}\label{smallrankupdate}
For each $n$ consider two symmetric $n\times n$ matrices $A_n$ and $C_n$. Assume that $\eta_{A_n}$ converges weakly to $\eta$ and $\rang(A_n-C_n)=o(n)$. Then $\eta_{C_n}$ converges weakly to $\eta$.
\end{lemma}
\begin{proof}
    By interlacing, for all $x\in \mathbb{R}$ we have that
    \[|\eta_{A_n}((-\infty,x])-\eta_{C_n}((-\infty,x])|\le \frac{\rang(A_n-C_n)}n=\frac{o(n)}{n}=o(1).\]
    Thus, for each continuity point $x$ of the cumulative distribution function $t\mapsto \eta((-\infty,t])$, we have
    \[\lim_{n\to\infty} \eta_{C_n}((-\infty,x])=\lim_{n\to\infty}\eta_{A_n}((-\infty,x])=\eta((-\infty,x]).
    \]
    Therefore, the statement follows.
\end{proof}
Combining \eqref{eqstep1}, the definition of $I$, the fact that $\varepsilon_n\to 0$ and Lemma~\ref{semidefcompare} below with the choice of $\delta_n=\sqrt{\varepsilon_n}$, we get that with probability $1$,
\begin{equation}\label{eqstep2}
\text{the empirical spectral measure of }d\sum_{i\in I}g_{\mathcal{T}_n,i}g_{\mathcal{T}_n,i}^T\text{ converges weakly to }\eta_k.
\end{equation}

\begin{lemma}\label{semidefcompare}
For each $n$ consider two nonnegative semidefinite $n\times n$ matrices $A_n$ and $C_n$. Assume that there is a sequence $\delta_n$ converging to $0$ such that $(1+\delta_n) C_n\ge A_n\ge (1-\delta_n) C_n$ and $\eta_{A_n}$ converges weakly to $\eta$. Then $\eta_{C_n}$ converges weakly to $\eta$.
\end{lemma}
\begin{proof}
    From the Courant-Fischer formula if for two nonnegative semidefinite $n\times n$ matrices $X$ and $Y$, we have $X\le Y$, then $\eta_X([0,x])\ge \eta_Y([0,x])$ for all $x\ge 0$.
    
    Thus, for any $x\ge 0$, we have
    \begin{align*}\eta_{C_n}([0,x])&\le \eta_{(1+\delta_n)^{-1} A_n}([0,x])=\eta_{ A_n}([0,(1+\delta_n)x])\text{ and }\\
    \eta_{C_n}([0,x])&\ge \eta_{(1-\delta_n)^{-1} A_n}([0,x])=\eta_{ A_n}([0,(1-\delta_n)x]).
    \end{align*}
    
    Thus, for each continuity point $x$ of the cumulative distribution function $t\mapsto \eta([0,t])$, we have
    \[\lim_{n\to\infty} \eta_{C_n}([0,x])=\lim_{n\to\infty}\eta_{A_n}([0,x])=\eta([0,x]).
    \]
    Therefore, the statement follows.
\end{proof}

Combining \eqref{eqstep2} with Lemma~\ref{smallrankupdate} and the fact that $|[r]\setminus I|=o(r)$, we obtain that 
with probability~$1$,
\begin{equation*}
\text{the empirical spectral measure of }d\sum_{i=1}^rg_{\mathcal{T}_n,i}g_{\mathcal{T}_n,i}^T\text{ converges weakly to }\eta_k,
\end{equation*}
in other words, $\psi_{\mathcal{T}_n}$ converges to $\eta_k$.
\end{proof}

Besides the weak limit of the measures $\psi_{\mathcal{T}_n}$, we also need good control over the behavior of  $\psi_{\mathcal{T}_n}$ near $0$. More precisely, we need the following uniform integrability condition. 

\begin{lemma}\label{lemmauniformint}
    Let $\mathcal{T}_n$ be drawn from the measure $\mathbb{P}^{H_n}$. Consider the coupling of $\mathcal{T}_1,\mathcal{T}_2,\dots$ provided by Theorem~\ref{thmNachmiasPeled}. Then with probability $1$, 
    \[\lim_{\delta\to 0}\limsup_{n\to\infty}\left|\int_{(0,\delta]}\log(t)\,\psi_{\mathcal{T}_n}(dt)\right|=0.\]
\end{lemma}
\begin{proof}

    Given a subset $\mathcal{T}$ of $U$, let $\lambda_{1,\mathcal{T}}\le \lambda_{2,\mathcal{T}}\le \cdots\le \lambda_{|\mathcal{T}|,\mathcal{T}}$ be the eigenvalues of $P_H\restriction\mathcal{T}$. It follows from \cite[Lemma 6]{meszaros2026homology} that if $\mathcal{T}$ is drawn from the measure $\mathbb{P}^H$, then for all $1\le \ell\le r$, we have
    
    \begin{equation}\label{Ebound}\mathbb{E}\prod_{i=1}^\ell (d\lambda_{i,\mathcal{T}})^{-1}\le \left(\frac{e^2mr}{d\ell^2}\right)^\ell.\end{equation}
    Using~\eqref{rvsn}, for large enough $n$, \[\frac{8mr}{d\ell^2}=\frac{r^2}{\ell^2}\frac{8m}{rd}\le \frac{r^2}{\ell^2}\frac{9m}{nd}=\frac{9r^2}{(k+1)\ell^2}\le \frac{9r^2}{\ell^2}.\]
    Combining this with Markov's inequality and \eqref{Ebound}, we get
    \begin{align*}\mathbb{P}\left(-\sum_{i=1}^\ell\log(d\lambda_{i,\mathcal{T}})\ge 2\log(n)+2\ell\log\left(\frac{3r}{\ell}\right)\right)&\le\mathbb{P}\left(-\sum_{i=1}^\ell\log(d\lambda_{i,\mathcal{T}})\ge 2\log(n)+\ell\log\left(\frac{8mr}{d\ell^2}\right)\right)\\&=\mathbb{P}\left(\prod_{i=1}^\ell (d\lambda_{i,\mathcal{T}})^{-1}\ge n^2\left(\frac{8mr}{d\ell^2}\right)^{\ell}\right)\\&\le \left(\frac{e^2mr}{d\ell^2}\right)^\ell n^{-2}\left(\frac{d\ell^2}{8mr}\right)^{\ell}\\&=n^{-2}\left(\frac{e^2}8\right)^\ell.
    \end{align*}

Combining this with the union bound, the probability that there is an index $1\le \ell\le r$ such that $-\sum_{i=1}^\ell\log(d\lambda_{i,\mathcal{T}})\ge 2\log(n)+2\ell\log\left(\frac{3r}{\ell}\right)$ is at most $n^{-2}\sum_{\ell=1}^\infty\left(\frac{e^2}8\right)^\ell<20n^{-2}$. 

Since $\sum_{n=1}^\infty 20n^{-2}<\infty$, by the Borel--Cantelli lemma, with probability $1$, for all large enough $n$, we have that
\[-\sum_{i=1}^\ell\log(d\lambda_{i,\mathcal{T}})\le 2\log(n)+2\ell\log\left(\frac{3r}{\ell}\right)\text{ for all }1\le \ell\le r.\]

Thus, from now on we assume that the event above occurs. Given $\delta$, let us define the random index $i_\delta=i_{\delta,n}$ as
\[i_\delta=|\{i\in[r]\,:\,d\lambda_{i,\mathcal{T}}\le \delta\}|.\]
We see that $i_\delta=r\psi_{\mathcal{T}_n}([0,\delta])$. We use the continuous value $0$ for $x\log(a/x)$ at $x=0$, for $a>0$. Since $\psi_{\mathcal{T}_n}$ converge weakly to $\eta_k$ by Lemma~\ref{psietaconv}, and $\lim_{\delta\to 0}\eta_k([0,\delta])=0$ by Lemma~\ref{etak0}, we see that for any fixed $0<\delta'<1$ provided that $\delta$ is small enough, we have $i_\delta\le r\delta'$ for all large enough $n$. Given such a $\delta$, for all large enough $n$, we have
\[\left|\int_{(0,\delta]} \log(t)\psi_{\mathcal{T}_n}(dt)\right|=-\frac{1}r\sum_{i=1}^{i_\delta} \log(d\lambda_{i,\mathcal{T}_n})\le \frac{2\log(n)}{r}+2\frac{i_\delta}{r}\log\left(\frac{3r}{i_\delta}\right)\le \frac{2\log(n)}{r}+2\delta'\log\left(\frac{3}{\delta'}\right). \]
Here the right hand side is less than $\varepsilon$ provided that $\delta'$ is small enough and $n$ is large enough.
\end{proof}
The following lemma was basically proved in \cite[Section 3.3]{meszaros2026homology}. For completeness, we provide a proof.
\begin{lemma}\label{logdetgeneral}
    Let $\mu_n$ be a sequence of probability measures on $[0,\infty)$ such that
    \begin{enumerate}[(a)]
        \item\label{conda} $\mu_n$ converge weakly to some probability measure $\mu$.
        \item\label{condb} We have
        \[\lim_{\delta\to 0}\limsup_{n\to\infty}\left|\int_{(0,\delta]}\log(t)\,\mu_{n}(dt)\right|=0.\]
        \item\label{condc} The sequence $\int t\,\mu_n(dt)$ is bounded by a finite constant $M$.
    \end{enumerate}
    Then
    \[\lim_{n\to\infty} \int_{(0,\infty)}\log(t) \mu_n(dt)= \int_{(0,\infty)}\log(t) \mu(dt).\]
    Moreover, the integral on the right is finite.
\end{lemma}
\begin{proof}
Given a $\delta\in (0,1)$, let us consider the function $f_\delta:\mathbb{R}\to\mathbb{R}$ defined as follows:
\[f_\delta(t)=\begin{cases}
0&\text{ if }t\le 0,\\
\frac{t}{\delta}\log(\delta)&\text{ if }0<t\le \delta,\\
\log(t)&\text{ if }\delta<t\le 1,\\
0&\text{ if }1<t.
\end{cases}\]

We have
\[\left|\int f_\delta(t)\,\mu_n(dt)-\int \mathbbm{1}(0<t<1)\log(t)\,\mu_n(dt) \right|\le \left|\int_{(0,\delta)}\log(t)\,\mu_n(dt)\right|.\]

Since $f_\delta$ is a bounded continuous function and $\mu_n$ converges weakly to $\mu$, we have
\[\limsup_{n\to \infty}\left|\int f_\delta(t)\,\mu_n(dt)-\int f_\delta\,\mu(dt) \right|=0.\]

Thus, for all large enough $n$, we have
\[\limsup_{n\to\infty}\left|\int f_\delta(t)\,\mu(dt)-\int \mathbbm{1}(0<t<1)\log(t)\,\mu_n(dt) \right|\le \limsup_{n\to\infty} \left|\int_{(0,\delta)}\log(t)\,\mu_n(dt)\right|.\]

Combining this with condition~\eqref{condb}, it follows that
\[\lim_{n\to\infty} \int \mathbbm{1}(0<t<1)\log(t)\,\mu_n(dt)=\lim_{\delta\to 0} \int f_\delta(t)\,\mu(dt)=\int \mathbbm{1}(0<t<1)\log(t)\,\mu(dt)>-\infty,\]
where the last equality follows from the monotone convergence theorem.

Given $\omega>e$, let us consider the function $g_\omega:\mathbb{R}\to\mathbb{R}$ defined as follows:
\[g_\omega(t)=\begin{cases}
0&\text{ if }t\le 1,\\
\log(t)&\text{ if }1<t\le\omega,\\
\log(\omega)&\text{ if }\omega<t.
\end{cases}\]

Then, by Jensen's inequality, we have
\begin{align*}\left|\int g_\omega(t)\,\mu_n(dt)-\int \mathbbm{1}(1<t)\log(t)\,\mu_n(dt) \right|&\le \int_{(\omega,\infty)}\log(t)\,\mu_n(dt)\\&\le \mu_n((\omega,\infty))\log\left(\frac{\int_{(\omega,\infty)} t \mu_n(dt)}{\mu_n((\omega,\infty))}\right)\\&\le \mu_n((\omega,\infty))\log\left(\frac{M}{\mu_n((\omega,\infty))}\right). 
\end{align*}

By Markov's inequality, $\mu_n((\omega,\infty))\le \frac{M}{\omega}$. The function $x\mapsto x\log\left(\frac{M}{x}\right)$ is monotone increasing on the interval $[0,\frac{M}{e}]$. Thus, recalling that $\omega>e$, we see that
\[\left|\int g_\omega(t)\,\mu_n(dt)-\int \mathbbm{1}(1<t)\log(t)\,\mu_n(dt) \right|\le \mu_n((\omega,\infty))\log\left(\frac{M}{\mu_n((\omega,\infty))}\right)\le \frac{M}{\omega}\log(\omega). \]

Since $g_\omega$ is a bounded continuous function and $\mu_n$ converges weakly to $\mu$, we have
\[\limsup_{n\to \infty}\left|\int g_\omega(t)\,\mu_n(dt)-\int g_\omega\,\mu(dt) \right|=0.\]

Thus,
\[\limsup_{n\to \infty}\left|\int \mathbbm{1}(1< t)\log(t)\,\mu_n(dt)-\int g_\omega\,\mu(dt) \right|\le \frac{M}{\omega}\log(\omega).\]
Since $\frac{M}{\omega}\log(\omega)\to 0$ as $\omega\to\infty$, it follows that
\[\lim_{n\to\infty} \int \mathbbm{1}(1\le t)\log(t)\,\mu_n(dt)=\lim_{\omega\to \infty} \int g_\omega(t)\,\mu(dt)=\int \mathbbm{1}(1\le t)\log(t)\,\mu(dt)<\infty.\qedhere\]
\end{proof}

\begin{lemma}\label{Tracebound}
Let $\mathcal{T}$ be an $r$-element subset of $U$. Then
\[\frac{d}{n}\Tr(P_H\restriction \mathcal{T})\le 8k.\]

\end{lemma}

\begin{proof}
    Let us choose $\varepsilon=\frac{1}4$, and define $I$ as in~\eqref{Idef}. Then it follows from \eqref{Ibound} that
    \[|[r]\setminus I|\le \frac{4kn}{d}.\]
    Let 
    \[P_1=\sum_{i\in I}g_ig_i^T\text{ and }P_2=\sum_{i\in [r]\setminus I}g_ig_i^T.\]

    Recall that for all $i\in I$, we have $\lambda_i\ge \frac{1}2$, so
    \[P_1=\sum_{i\in I}g_ig_i^T\le 2 \sum_{i=1}^r\lambda_ig_ig_i^T=\frac{2}dB^TB.\]
    Each diagonal entry of $B^TB$ is $k+1$. Therefore, each diagonal entry of $P_1$ is at most $\frac{2(k+1)}d$. Thus,
    \[\Tr(P_1\restriction \mathcal{T})\le \frac{2(k+1)n}d.\]
    Then,
    \[\Tr(P_H\restriction \mathcal{T})=\Tr(P_1\restriction \mathcal{T})+\Tr(P_2\restriction \mathcal{T})\le \frac{2(k+1)n}d+\Tr(P_2)=\frac{2(k+1)n}d+|[r]\setminus I|\le \frac{8kn}{d}.\qedhere\]
\end{proof}

Now we are ready to prove Theorem~\ref{thm1}.

We have
\begin{equation}\label{Pasintlog}\frac{1}r\log\mathbb{P}^{H_n}(\mathcal{T}_n)=\frac{1}r\log\det (P_H\restriction\mathcal{T}_n)=\frac{1}r\log \left(d^{-r}\det(dP_H\restriction\mathcal{T}_n)\right)=-\log(d)+\int \log(t) \,\psi_{\mathcal{T}_n}(dt).\end{equation}

Let us consider the coupling of $\mathcal{T}_1,\mathcal{T}_2,\dots$ provided by Theorem~\ref{thmNachmiasPeled}. We claim that with probability one, the sequence of measures $\mu_n=\psi_{\mathcal{T}_n}$ satisfies the conditions of Lemma~\ref{logdetgeneral} with the choice of $\mu=\eta_k$. Indeed, condition~\eqref{conda} is clear from Lemma~\ref{psietaconv}, condition~\eqref{condb} follows from Lemma~\ref{lemmauniformint}. To prove condition~\eqref{condc}, combining Lemma~\ref{Tracebound} and the fact that $\frac{r}n\to 1$, we have 
\[\limsup_{n\to\infty} \int t\mu_n(dt)= \limsup_{n\to\infty} \frac{d}{r} \Tr(P_H\restriction\mathcal{T}_n)\le 8k.\]

Thus, by Lemma~\ref{logdetgeneral}, we see that with probability $1$,
\[\lim_{n\to\infty}\int \log(t) \,\psi_{\mathcal{T}_n}(dt)=\int \log(t) \,\eta_k(dt)=2c_k,\]
where $c_k$ was defined in Definition~\ref{ckdef}. 

Combining this with~\eqref{Pasintlog}, we obtain, under this coupling,
\begin{equation}\label{AEPrank}
    \frac{\log\mathbb{P}^{H_n}(\mathcal{T}_n)}r+\log d
    \longrightarrow2c_k\qquad\text{almost surely.}
\end{equation}
Finally,
\[
    \frac{\log\mathbb{P}^{H_n}(\mathcal{T}_n)}n+\log d
    =\frac rn\left(\frac{\log\mathbb{P}^{H_n}(\mathcal{T}_n)}r+\log d\right)
    +\left(1-\frac rn\right)\log d\longrightarrow2c_k
\]
almost surely, by~\eqref{ranklognormalization}. This proves Theorem~\ref{thm1}.

\section{Coupling in the general setting -- Proof of Theorem~\ref{thmgeneralcoupling}}\label{SecGeneralCoupling}

\subsection{Constructing the coupling}
Recall that $R(v)$ denotes the independent neighbor choice of a vertex $v$ defining $\mathcal{T}_n^{1o}$.
Let $\mathcal{F}_n$ be the set of functions $f:V_n\to U_n\cup\{*\}$ such that for all $v\in V_n$, $f(v)=*$  or $f(v)$ is a neighbor of $v$ in the graph $\mathfrak{G}_n$. By considering the edge set $\{(vf(v))\,:\,v\in V_n,\, f(v)\neq *\}$, we see that there is a one-to-one correspondence between $\mathcal{F}_n$ and the set of subgraphs of $\mathfrak{G}_n$ where each vertex of $V_n$ has degree at most~$1$.

Fix $0<\varepsilon<1/2$. 
Let $X_n$ be a random element of $\mathcal{F}_n$ such that $(X_n(v))_{v\in V_n}$ are independent and
\[\mathbb{P}(X_n(v)=u)=\begin{cases}
\varepsilon&\text{ if $u=*$,}\\
\frac{1-\varepsilon}d&\text{ if $u$ is a neighbor of $v$,}\\
0&\text{ otherwise.}
\end{cases}\]

Another way to get $X_n$ is as follows: Let $(I_{v})_{v\in V_n}$ be i.i.d. Bernoulli($\varepsilon$) random variables independent of $(R(v))_{v\in V_n}$, and let
\[X_n(v)=\begin{cases}*&\text{ if $I_v=1$,}\\
R(v)&\text{ if $I_v=0$.}
\end{cases}\]
Recall that $\mathcal{T}_n^{1o}=\{R(v)\,:\,v\in V_n\}$. Thus, the definition above gives a coupling of $\mathcal{T}^{1o}_n$ and $X_n$ such that
\begin{equation}\label{ToutvsX}\mathbb{E}|\mathcal{T}^{1o}_n\Delta \{X_n(v)\,:\,v\in V_n,\,X_n(v)\neq*\}|\le\varepsilon n.
\end{equation}

Let $\mathcal{T}$ be a subset of $U_n$, and let $\mathcal{M}=\mathcal{M}(\mathcal{T})$ be the set of matchings of the graph $\mathfrak{G}_n[\mathcal{T}]$. For an $M\in \mathcal{M}$, let $|M|$ be the number of edges of $M$. Consider the measure $\mathbb{P}^\mathcal{T}$ on $\mathcal{M}$ where the probability assigned to a matching $M$ is proportional to $(1-\varepsilon)^{|M|}\varepsilon^{n-|M|}$. In other words,

\[\mathbb{P}^\mathcal{T}(M)=\frac{(1-\varepsilon)^{|M|}\varepsilon^{n-|M|}}{Z_\varepsilon(\mathcal{T})},\]
where 
\[Z_\varepsilon(\mathcal{T})=\sum_{M\in \mathcal{M}} (1-\varepsilon)^{|M|}\varepsilon^{n-|M|}.\]

\begin{lemma}\label{lemmaZlimit}
    Let $\mathcal{T}_n$ be an $r$-element subset of $U_n$ such that $(\mathfrak{G}_n[\mathcal{T}_n],o_n)$ converges to $(\mathbb{T}_k,o)$, where $o_n$ is a uniform random vertex in $V_n$ and $(\mathbb{T}_k,o)$ is the random rooted tree from the statement of Theorem~\ref{thmNachmiasPeled}. Let $\eta_k$ be defined as in Definition~\ref{ckdef}. Then
    \[\lim_{n\to\infty}\frac{\log Z_\varepsilon(\mathcal{T}_n)}n=\zeta_\varepsilon , \]
    where 
    \[\zeta_\varepsilon= \int\log(\varepsilon+(1-\varepsilon)x) \eta_k(dx). \]

    Moreover,
    \begin{equation}\label{zetalimit}\lim_{\varepsilon\to 0} \zeta_\varepsilon=\int_{(0,\infty)} \log(x)\, \eta_k(dx)=2c_k.\end{equation}
\end{lemma}
\begin{proof}
Let 
\[f_n(t)=\sum_{M\in \mathcal{M}(\mathcal{T}_n)} t^{|M|}.\]
Note that $Z_\varepsilon(\mathcal{T}_n)=\varepsilon^n f_n\left(\frac{1-\varepsilon}\varepsilon\right)$.

Let $\rho_n$ be the matching measure of $\mathfrak{G}_n[\mathcal{T}_n]$. Then as in the proof of \cite[Theorem 3.3.]{abert2015matching}, we see that
\[\frac{1}{n+r}\log f_n\left(\frac{1-\varepsilon}\varepsilon\right)=\frac{1}2\int \log\left(1+\frac{1-\varepsilon}\varepsilon x^2\right) \rho_n(dx). \]

By Lemma~\ref{mathingmeasureconvergence}, we see that $\rho_n$ converges weakly to $\nu=\mathbb{E}\nu_{(\mathbb{T}_k,o)}$. Combining this with the fact that $r=(1+o(1))n$, and Lemma~\ref{mathingmeasureconvergence}, we obtain that
\[\lim_{n\to\infty}\frac{1}{n}\log f_n\left(\frac{1-\varepsilon}\varepsilon\right)=\lim_{n\to\infty}\int \log\left(1+\frac{1-\varepsilon}\varepsilon x^2\right) \rho_n(dx)=\int \log\left(1+\frac{1-\varepsilon}\varepsilon x^2\right) \nu(dx).\]

(A small technical problem is that the integrand is not bounded. But it grows very slowly, and we have a uniform bound on the second moment of $\rho_n$, thus, the convergence above indeed holds. See Section~\ref{secsectechdetails} for the technical details.)

Thus,
\[\lim_{n\to\infty}\frac{\log Z_\varepsilon(\mathcal{T}_n)}n=\log(\varepsilon)+\int \log\left(1+\frac{1-\varepsilon}\varepsilon x^2\right) \nu(dx)=\int\log(\varepsilon+(1-\varepsilon)x) \eta_k(dx).\]

Finally,~\eqref{zetalimit} follows by applying the monotone convergence theorem separately to the negatives of the integrands on $(0,1)$ and to the integrands on $[1,\infty)$.
\end{proof}

Let $\mathcal{T}_n$ be drawn from the measure $\mathbb{P}^{H_n}$. Then combining Theorem~\ref{thm1}, Theorem~\ref{thmNachmiasPeled}, Lemma~\ref{lemmaZlimit} and~\eqref{ranklognormalization}, we see that there is an event $\mathcal{A}_n$ such that $\mathbb{P}^{H_n}(\mathcal{A}_n)\to 1$ and on the event $\mathcal{A}_n$
\[\left|\frac{\log Z_\varepsilon(\mathcal{T}_n)}n-\zeta_\varepsilon\right|<\varepsilon \text{ and }\left|\frac{\log\mathbb{P}^{H_n}(\mathcal{T}_n)}n+\log(d)-2c_k\right|<\frac{\varepsilon}2.\]

For all sufficiently large $n$, we have $\mathbb{P}^{H_n}(\mathcal{A}_n)\geq\exp(-\varepsilon n/2)$. For such $n$, let $\mathcal{T}_n'$ be distributed as $\mathcal{T}_n$ conditioned on $\mathcal{A}_n$. Then almost surely,
\[Z_\varepsilon(\mathcal{T}_n')\ge \exp(n(\zeta_\varepsilon-\varepsilon)).\]
Conditioning divides point probabilities by $\mathbb{P}^{H_n}(\mathcal{A}_n)$, so $\mathcal{T}_n'$ takes any given value with probability at most $d^{-n}\exp(n(2c_k+\varepsilon))$.

Let the matching $M_n$ be drawn from $\mathbb{P}^{\mathcal{T}_n'}$, and let $Y_n$ be the corresponding element of $\mathcal{F}_n$. Clearly, $|M_n|\le r$.

Note that $X_n$ and $Y_n$ both depend on $\varepsilon$, but to ease the notation, we do not indicate this dependence in the notation.

For all sufficiently large $n$, $m=nd/(k+1)\geq2n$. For $a>0$, the expressions $x\log(a/x)$ and $(a/x)^x$ are assigned their continuous values $0$ and $1$, respectively, at $x=0$. Now let $M$ be a fixed matching of $\mathfrak{G}_n$ with at most $r$ edges and let $F_M$ be the corresponding element of $\mathcal{F}_n$. Let us estimate the probability of the event that $Y_n=F_M$. Given that $Y_n=F_M$ the number of possibilities for the value of $\mathcal{T}_n'$ is at most
\[{{m}\choose{r-|M|}}\le {{m}\choose{n-|M|}}\le \left(\frac{em}{n-|M|}\right)^{n-|M|}.\]
Therefore,
\[\mathbb{P}(Y_n=F_M)\le \left(\frac{em}{n-|M|}\right)^{n-|M|} d^{-n}\exp(n(2c_k+\varepsilon))\exp(-n(\zeta_\varepsilon-\varepsilon))(1-\varepsilon)^{|M|}\varepsilon^{n-|M|}.\]
Also,
\[\mathbb{P}(X_n=F_M)=\frac{(1-\varepsilon)^{|M|}\varepsilon^{n-|M|}}{ d^{|M|}}.\]

Therefore,
\begin{align*}\log\frac{\mathbb{P}(Y_n=F_M)}{\mathbb{P}(X_n=F_M)}&\le \log\left(\left(\frac{em}{d(n-|M|)}\right)^{n-|M|} \exp(n(2c_k+\varepsilon))\exp(-n(\zeta_\varepsilon-\varepsilon))\right)\\&=(n-|M|)\log\left(\frac{en}{(k+1)(n-|M|)}\right)+2\varepsilon n+(2c_k-\zeta_\varepsilon)n\end{align*}

Thus,
\begin{equation}\label{dklbound}\frac{\dkl(Y_n,X_n)}n\le 2\varepsilon+2c_k-\zeta_\varepsilon+\mathbb{E} Q_n\log\left(\frac{e}{(k+1)Q_n} \right),\end{equation}
where $Q_n=\frac{n-|M_n|}n$.

Let $0<\eta<1$.

Each vertex in $V(\mathfrak{G}_n[\mathcal{T}_n'])\cap U_n$ has degree $k+1$. Thus, it follows easily that $\mathfrak{G}_n[\mathcal{T}_n']$ has at most $(k+2)^r$ matchings. If $\mathcal{T}_n'$ appears with positive probability then $B_n[\mathcal{T}_n']$ has an $r\times r$ submatrix with non-zero determinant, so $\mathfrak{G}_n[\mathcal{T}_n']$ must have a matching of size $r$ by the Leibniz formula for determinants. Thus, $Z_\varepsilon(\mathcal{T}_n')\ge (1-\varepsilon)^r\varepsilon^{n-r}$. 

Therefore, we have
\[\mathbb{P}(Q_n>\eta)\le \frac{(k+2)^r(1-\varepsilon)^{n-\eta n}\varepsilon^{\eta n}}{Z_\varepsilon(\mathcal{T}_n')}\le (2(k+2))^r \varepsilon^{\eta n +r-n}\le (2(k+2))^r \varepsilon^{\frac{\eta}2 r}\]
provided that $n$ is large enough.

Thus, $\mathbb{P}(Q_n>\eta)$ tends to $0$ if $\varepsilon<(2(k+2))^{-\frac{2}{\eta}}$. Thus,
\begin{align}\lim_{\varepsilon\to 0}\limsup_{n\to\infty}\mathbb{E} Q_n\log\left(\frac{e}{(k+1) Q_n}\right)&=0\text{ and }\label{QlogQ}\\\lim_{\varepsilon\to 0}\limsup_{n\to\infty}\mathbb{E} Q_n&=0\label{limEQn}.
\end{align}

Therefore, combining \eqref{dklbound}, \eqref{zetalimit} and \eqref{QlogQ}, it follows that
\[\lim_{\varepsilon\to 0}\limsup_{n\to\infty}\frac{\dkl(Y_n,X_n)}n=0.\]

Thus, by Lemma~\ref{lemmaMarton}, there is a coupling of $X_n$ and $Y_n$ such that
\begin{equation}\label{XvsY}\lim_{\varepsilon\to 0}\limsup_{n\to\infty}\frac{1}n \sum_{v\in V_n}\mathbb{P}(X_n(v)\neq Y_n(v))=0.\end{equation}

Let us recall that the definition of $Y_n$ also provides us a coupling of $Y_n$ and $\mathcal{T}_n'$. Moreover, we also defined a joint distribution of $X_n$ and $\mathcal{T}_n^{1o}$. Now these couplings can be combined into a single coupling $(\mathcal{T}_n^{1o},X_n,Y_n,\mathcal{T}_n')$ such that all the pairs $(\mathcal{T}_n^{1o},X_n)$, $(X_n,Y_n)$ and $(Y_n,\mathcal{T}_n')$ have the above-mentioned joint distributions. Then under this coupling
\[\mathbb{E}\frac{|\mathcal{T}_n^{1o}\Delta \mathcal{T}_n'|}n\le \mathbb{E} \frac{|\mathcal{T}_n^{1o}\Delta \{X_n(v)\,:\,v\in V_n,\,X_n(v)\neq*\}|}n+\frac{2}n \sum_{v\in V_n}\mathbb{P}(X_n(v)\neq Y_n(v))+\mathbb{E}Q_n.\]
Here the right hand side can be made arbitrarily small by choosing $\varepsilon$ small enough and $n$ large enough, as we see by combining \eqref{ToutvsX}, \eqref{XvsY} and \eqref{limEQn}.

Finally, on $\mathcal{A}_n$ use the constructed coupling of $(\mathcal{T}_n',\mathcal{T}_n^{1o})$. On $\mathcal{A}_n^c$, draw $\mathcal{T}_n$ from its original conditional law and couple it independently with the usual $1$-out law. This preserves both marginals. Since $|\mathcal{T}_n|=r\leq n$ and $|\mathcal{T}_n^{1o}|\leq n$, the contribution from $\mathcal{A}_n^c$ to the normalized expected symmetric difference is at most $2\mathbb{P}(\mathcal{A}_n^c)\to0$. Thus
\[\mathbb{E}\frac{|\mathcal{T}_n^{1o}\Delta\mathcal{T}_n|}n\]
can be made arbitrarily small by choosing $\varepsilon$ small enough and then $n$ large enough. A diagonal choice $\varepsilon=\varepsilon_n\to0$ gives the sequence of couplings asserted in Theorem~\ref{thmgeneralcoupling}. For the remaining finitely many indices, choose any coupling.

\subsection{Technical details of the proof of Lemma~\ref{lemmaZlimit}}\label{secsectechdetails}

Let $G_n=\mathfrak{G}_n[\mathcal{T}_n]$. 
First note that, by Lemma~\ref{lemmadisintegrate} and the fact that $r=n-o(n)$, we have
\begin{align*}\int t^2 \rho_n(dt)&=(1+o(1))\frac{1}n\sum_{v\in V_n}\int t^2 \nu_{(T_v(G_n),\bar{v})}(dt)\\&= (1+o(1))\frac{1}n\sum_{v\in V_n} \langle A_{T_v(G_n)}^2\chi_{\bar{v}},\chi_{\bar{v}}\rangle\\&=(1+o(1))\frac{1}n\sum_{v\in V_n}\deg_{G_n}(v)\\&=(1+o(1))\frac{|E(G_n)|}n\\&=(1+o(1))(k+1)<k+2 \end{align*}
for all large enough $n$, which we will assume from now on. 

Although the rest of the proof follows closely the second part of the proof of Lemma~\ref{logdetgeneral}, for the reader's convenience, we provide the details.

Let $\mu_n$ be the pushforward of $\rho_n$ by the map $t\mapsto t^2$. Then the measure $\mu_n$ is supported on $[0,\infty)$, $\int t\, \mu_n(dt)<k+2$ and $\mu_n$ converges weakly to $\eta_k$. Moreover, \[\int \log\left(1+\frac{1-\varepsilon}\varepsilon x^2\right)\,\rho_n(dx)=\int_{[0,\infty)} g(t) \,\mu_n(dt),\] 
where $g(t)$ is defined on $[0,\infty)$ as \[g(t)=\log\left(1+\frac{1-\varepsilon}\varepsilon t\right).\]

Given $\omega>0$, let us consider the function $g_\omega:[0,\infty)\to\mathbb{R}$ defined as follows:
\[g_\omega(x)=\begin{cases}
g(x)&\text{ if }x\le \omega,\\
g(\omega)&\text{ if }x>\omega.
\end{cases}\]

Then, by Jensen's inequality, we have
\begin{align*}\left|\int g_\omega(t)\,\mu_n(dt)-\int g(t)\,\mu_n(dt) \right|&\le \int_{(\omega,\infty)}g(t)\,\mu_n(dt)\\&\le \mu_n((\omega,\infty))g\left(\frac{\int_{(\omega,\infty)} t \mu_n(dt)}{\mu_n((\omega,\infty))}\right)\\&\le \mu_n((\omega,\infty))g\left(\frac{k+2}{\mu_n((\omega,\infty))}\right). 
\end{align*}

By Markov's inequality, $\mu_n((\omega,\infty))\le \frac{k+2}{\omega}$. The function $x\mapsto x g\left(\frac{k+2}{x}\right)$ is increasing on the interval $[0,\infty)$. Thus, 

\[\left|\int_{[0,\infty)} g_\omega(t)\,\mu_n(dt)-\int_{[0,\infty)} g(t)\,\mu_n(dt) \right|\le \frac{k+2}{\omega} g(\omega).\]

Since $g_\omega$ is a bounded continuous function and $\mu_n$ converges weakly to $\eta_k$, we have
\[\limsup_{n\to \infty}\left|\int g_\omega(t)\,\mu_n(dt)-\int g_\omega(t)\,\eta_k(dt) \right|=0.\]

Thus,
\[\limsup_{n\to \infty}\left|\int g(t)\,\mu_n(dt)-\int g_\omega\,\eta_k(dt) \right|\le \frac{k+2}{\omega}g(\omega).\]
Since $\frac{k+2}{\omega}g(\omega)\to 0$ as $\omega\to\infty$, it follows that
\[\lim_{n\to\infty} \int g(t)\,\mu_n(dt)=\lim_{\omega\to \infty} \int g_\omega(t)\,\eta_k(dt)=\int g(t)\,\eta_k(dt)=\int \log\left(1+\frac{1-\varepsilon}\varepsilon x^2\right) \nu(dx).\]

\medskip

\section{Corollaries of Theorem~\ref{thm1}}\label{SecAEPCorollaries}

The following lemma will be useful in the proofs below.

\begin{lemma}\label{lemmaBTBvsP}
For any $r$-element subset $\mathcal{T}\subset U$, we have
\[\det(B^TB\restriction \mathcal{T})=\mathbb{P}^H(\mathcal{T}){\det}^+(BB^T).\]

\end{lemma}
\begin{proof}
Let $C$ be an $r\times m$ matrix such that the rows of $C$ form an orthonormal basis of $H$. Then we can find an $n\times r$ matrix $S$ such that $B=SC$. As the rows of $B$ are indexed by $V$, we may also assume that the rows of $S$ are indexed by $V$. Given a matrix $A$, a subset $X$ of the rows and a subset $Y$ of the columns, the corresponding submatrix is denoted by $A[X,Y]$. If $X$ is the set of all the rows, we will sometimes use the notation $A[*,Y]$. If $Y$ is the set of all the columns, we will sometimes use the notation $A[X,*]$. Let $V_r$ be the family of $r$-element subsets of $V$.

By applying the Cauchy-Binet formula, we see that
\begin{align*}\det(B^TB\restriction \mathcal{T})&=\sum_{X\in V_r} (\det B[X,\mathcal{T}])^2\\&=\sum_{X\in V_r} (\det S[X,*] \det C[*,\mathcal{T}])^2\\&=(\det C[*,\mathcal{T}])^2\sum_{X\in V_r}\det ((SS^T)[X,X]).\end{align*}

Here
\[(\det C[*,\mathcal{T}])^2=\det\big((C[*,\mathcal{T}])^T C[*,\mathcal{T}]\big)=\det P_H\restriction \mathcal{T}=\mathbb{P}^H(\mathcal{T}).\]

Moreover, noting that $CC^T=I$, we have
\[\sum_{X\in V_r}\det ((SS^T)[X,X])=\sum_{X\in V_r}\det (SCC^TS^T)[X,X]=\sum_{X\in V_r} \det BB^T[X,X]={\det}^+(BB^T)\]
by the fact that $BB^T$ has rank $r$. Thus, the statement follows.
\end{proof}

\subsection{Proof of Corollary~\ref{CorNumberOfTrees}}

Let $B_n$ be a signed incidence matrix of $G_n$. The determinantal process associated with its row space is a uniform spanning tree, so $\mathbb{P}^{H_n}(\mathcal{T}_n)=1/\tau(G_n)$. Theorem~\ref{thm1} gives
\[
    -\frac{\log\tau(G_n)}n+\log d\longrightarrow2c_1\,.
\]
Using Cayley's formula for the complete graph shows that $c_1=0$. Thus the statement follows.

\subsection{The proof of Corollary~\ref{corEntropy}}
We start with the following corollary of Lemma~\ref{lemmaBTBvsP}.
\begin{lemma}\label{bounded}
Let $\mathcal{T}$ and $\mathcal{T}'$ be two $r$-element subsets of $U$, such that $\mathbb{P}^H(\mathcal{T})$ and $\mathbb{P}^H(\mathcal{T}')$ are both positive. Then
\[\left|\frac{\log\mathbb{P}^H(\mathcal{T})}r-\frac{\log\mathbb{P}^H(\mathcal{T}')}r\right|\le \log(k+1).\]
\end{lemma}
\begin{proof}
    By Lemma~\ref{lemmaBTBvsP}, we have
    \[\frac{\mathbb{P}^H(\mathcal{T})}{\mathbb{P}^H(\mathcal{T}')}=\frac{\det(B^TB\restriction \mathcal{T})}{\det(B^TB\restriction \mathcal{T}')}.\]
    The matrix $B^TB\restriction \mathcal{T}$ has integer entries, so $\det(B^TB\restriction \mathcal{T})\ge 1$. Moreover, all the diagonal entries of $B^TB\restriction \mathcal{T}$ are equal to $k+1$, so by Hadamard's inequality $\det(B^TB\restriction \mathcal{T})\le (k+1)^r$. The same is true for $\mathcal{T}'$ in place of~$\mathcal{T}$. So
    \[(k+1)^{-r}\le \frac{\mathbb{P}^H(\mathcal{T})}{\mathbb{P}^H(\mathcal{T}')}\le (k+1)^r.\]
    Taking logarithms, the statement follows.
\end{proof}

Let $\mathcal{T}_n$ be drawn from the measure $\mathbb{P}^{H_n}$. Let us consider the coupling of $\mathcal{T}_1,\mathcal{T}_2,\dots$ provided by Theorem~\ref{thmNachmiasPeled}. Let
\[X_n=-\frac{\log \mathbb{P}^{H_n}(\mathcal{T}_n)}{n}-\log(d_n)+2c_k.\]
In Section~\ref{proofofthm1}, we proved that $X_n$ converges to $0$ almost surely. For all sufficiently large $n$, its support therefore contains a value in $[-1,1]$. By Lemma~\ref{bounded} and $r\leq n$, the diameter of this support is at most $\log(k+1)$. Hence $|X_n|\leq1+\log(k+1)$ almost surely for all sufficiently large $n$. The dominated convergence theorem gives
\[\lim_{n\to\infty}\mathbb{E}X_n=0.\]
The statement follows from
\[\mathbb{E}X_n=\frac{H(\mathcal{T}_n)}n-\log(d_n)+2c_k.\]

\subsection{Gram determinants}

The following corollary will be useful in our study of homology torsion growth in the next section. Given a square matrix $A$, let ${\det}^+(A)$ denote the product of its nonzero eigenvalues.

\begin{corollary}\label{thmBTB}
    In the setup of Theorem~\ref{thm1}, assume that
    \begin{equation}\label{niceB}
        \lim_{j\to\infty}\frac{\log{\det}^+(d^{-1}B_jB_j^T)}{r}=0\,.
    \end{equation}
    Then, as $j\to\infty$,
    \[
        \frac{\log\det(B_j^TB_j\restriction\mathcal{T}_j)}{r}
        \text{ converges to }2c_k\text{ in probability.}
    \]
    Moreover, let $\mu_j$ be the uniform probability measure on the nonzero eigenvalues of $d^{-1}B_jB_j^T$, counted with multiplicity. Then~\eqref{niceB} is implied by the uniform integrability condition
    \begin{equation}\label{niceB2}
        \lim_{\delta\downarrow0}\limsup_{j\to\infty}
        \int_{(0,\delta)}|\log t|\,\mu_j(dt)=0\,.
    \end{equation}
\end{corollary}

In particular, condition~\eqref{niceB2} holds if the smallest nonzero eigenvalue of $d^{-1}B_jB_j^T$ is bounded away from $0$ uniformly in $j$.

\begin{proof}
By Lemma~\ref{lemmaBTBvsP},
\[
    \frac{\log\det(B_j^TB_j\restriction\mathcal{T}_j)}{r}
    =\frac{\log\mathbb{P}^{H_j}(\mathcal{T}_j)}{r}+\log d
    +\frac{\log{\det}^+(d^{-1}B_jB_j^T)}{r}\,.
\]
The first two terms on the right converge to $2c_k$ in probability by~\eqref{AEPrank}, and the last term tends to $0$ by~\eqref{niceB}. This proves the first statement.

As shown in the proof of Theorem~\ref{thm1}, the measures $\mu_j$ converge weakly to $\delta_1$, and $r/|V_j|\to1$. Since each diagonal entry of $B_jB_j^T$ equals $d$, we also have
\[
    \int t\,\mu_j(dt)=\frac{\Tr(d^{-1}B_jB_j^T)}{r}
    =\frac{|V_j|}{r}\longrightarrow1\,.
\]
Combining these observations with~\eqref{niceB2}, Lemma~\ref{logdetgeneral} gives
\[
    \frac{\log{\det}^+(d^{-1}B_jB_j^T)}{r}
    =\int\log t\,\mu_j(dt)\longrightarrow0\,.
\]
Thus~\eqref{niceB2} implies~\eqref{niceB}.
\end{proof}

\section{Homology torsion growth}\label{SecTop}

We first prove Theorem~\ref{thmBTB2} in the polytopal setting of Section~\ref{SecTorsionIntroduction}. We then use the cellular matrix-tree theorem to give spectral sufficient conditions for the normalization assumption~\eqref{nicetau}.

\subsection{Proof of Theorem~\ref{thmBTB2}}

Set $n_j=f_{\ell-1}(\Delta_j)$. The boundary matrices satisfy the setup of Theorem~\ref{thm1}, with $|V_j|=n_j$. Hence
\[
    \frac{\log\mathbb{P}^{H_j}(\mathcal{T}_j)}{n_j}+\log d_j
    \text{ converges to }2c_k\text{ in probability.}
\]
Recall that $t_{\ell-1}(S)=|\tors H_{\ell-1}(S,\mathbb{Z})|$. Identifying a hyperforest with its set of $\ell$-cells, the determinantal weight formula gives
\[
    \frac{2\log t_{\ell-1}(\mathcal{T}_j)}{n_j}
    =\left(\frac{\log\mathbb{P}^{H_j}(\mathcal{T}_j)}{n_j}+\log d_j\right)
    +\left(\frac{\log\tau_\ell(\Delta_j)}{n_j}-\log d_j\right)\,.
\]
The first parenthesized term converges to $2c_k$ in probability, and the second tends to $0$ by~\eqref{nicetau}. This proves the theorem.

\subsection{Spectral sufficient conditions}

An $\ell$-dimensional complex $\Delta$ is called \emph{acyclic in positive codimension}, or APC, if $\widetilde H_i(\Delta,\mathbb{Q})=0$ for every $i<\ell$. In this case, a spanning $\ell$-hyperforest has finite reduced homology in dimension $\ell-1$ and is also called a \emph{cellular spanning tree}. If $\Delta$ is an APC, then for $1\leq p\leq\ell$ and any the spanning $p$-hyperforests $S$ of the $p$-skeleton of~$\Delta$, we have $\tors \widetilde H_{p-1}(S,\mathbb{Z})=\widetilde H_{p-1}(S,\mathbb{Z})$. Thus, for all $1\le p\le \ell$, $\tau_p(\Delta)$ is the sum of $|\widetilde H_{p-1}(S,\mathbb{Z})|^2$ over the spanning $p$-hyperforests of the $p$-skeleton of $\Delta$.

Set $n_j=f_{\ell-1}(\Delta_j)$ and $r_j=\rang B_j$. By \eqref{ranklognormalization},
\begin{equation}\label{ranklognormalization2}
    \frac{r_j}{n_j}\longrightarrow1,
    \qquad
    \left(1-\frac{r_j}{n_j}\right)\log d_j\longrightarrow0\,.
\end{equation}
\begin{lemma}\label{matrixtreeequivalence}
    In the polytopal setup of Theorem~\ref{thmBTB2}, assume that $\Delta_j$ is APC and that
    \[
        \frac{\log f_{\ell-1}(\Delta_j)}{d_j}\longrightarrow0\,.
    \]
    Then
    \begin{equation}\label{topologicalcomparison}
        \lim_{j\to\infty}\frac{1}{f_{\ell-1}(\Delta_j)}
        \log\left(\frac{{\det}^+(B_jB_j^T)}{\tau_\ell(\Delta_j)}\right)=0\,.
    \end{equation}
    Consequently, the normalization conditions~\eqref{nicetau} and~\eqref{niceB} are equivalent.
\end{lemma}

\begin{proof}
If $\ell=1$, then $\Delta_j$ is a connected graph and the usual matrix-tree theorem gives
\[
    \frac{{\det}^+(B_jB_j^T)}{\tau_1(\Delta_j)}=f_0(\Delta_j)\,.
\]
Thus~\eqref{topologicalcomparison} holds. We assume for the rest of the proof that $\ell\geq2$ and drop the index $j$. Write $n=f_{\ell-1}(\Delta)$ and $r=\rang B$.

Let $C=\partial_{\ell-1}$ be the boundary matrix in dimension $\ell-1$, and set $h=\rang C$. The cellular matrix-tree theorem~\cite[Theorem~2.8]{duval2011cellular} gives
\begin{equation}\label{matrixtree}
    {\det}^+(BB^T)
    =\frac{\tau_\ell(\Delta)\tau_{\ell-1}(\Delta)}
    {|\widetilde H_{\ell-2}(\Delta,\mathbb{Z})|^2}\,.
\end{equation}
For every spanning $(\ell-1)$-hyperforest $S$ of the $(\ell-1)$-skeleton, the inclusion induces a surjection from $\widetilde H_{\ell-2}(S,\mathbb{Z})$ onto $\widetilde H_{\ell-2}(\Delta,\mathbb{Z})$. Both groups are finite, so
\[
    \tau_{\ell-1}(\Delta)
    \geq |\widetilde H_{\ell-2}(\Delta,\mathbb{Z})|^2\,.
\]
Moreover, the integer lattice identity in~\cite[Lemma~3.2]{lyons2009random} gives
\[
    | \widetilde H_{\ell-2}(S,\mathbb{Z})|^2
    \leq\det(C^TC\restriction F_{\ell-1}(S))\,.
\]
Summing over $S$ and applying the Cauchy--Binet formula, we obtain
\[
    \tau_{\ell-1}(\Delta)\leq{\det}^+(CC^T)\,.
\]
Combining these bounds with~\eqref{matrixtree},
\begin{equation}\label{taubounds}
    1\leq\frac{{\det}^+(BB^T)}{\tau_\ell(\Delta)}
    \leq\tau_{\ell-1}(\Delta)\leq{\det}^+(CC^T)\,.
\end{equation}

We next bound the rank and trace of $CC^T$. Every $(\ell-1)$-cell $F$ is a facet of an $\ell$-cell $P$. Each $(\ell-2)$-face of $F$ belongs to exactly two facets of $P$, and distinct such faces determine distinct facets of $P$ other than $F$. Since $P$ has $k+1$ facets, $F$ has at most $k$ faces of dimension $\ell-2$. Therefore,
\[
    \Tr(CC^T)\leq kn\,.
\]

Let $a$ be the number of $(\ell-2)$-cells that are contained in at least one $(\ell-1)$-cell. Fix such a cell $\sigma$ and a cell $F$ containing it. For each of the $d$ top-dimensional cells containing $F$, there is another facet containing $\sigma$. These other facets are distinct, since otherwise a pair of $(\ell-1)$-cells would belong to more than one $\ell$-cell. Thus $\sigma$ belongs to at least $d+1$ cells of dimension $\ell-1$. Double counting incidences gives
\[
    a(d+1)\leq kn\,.
\]
The rows of $C$ corresponding to the remaining $(\ell-2)$-cells are zero, so
\[
    h\leq a\leq\frac{kn}{d+1}\,.
\]
Each nonzero eigenvalue of $CC^T$ is at most its trace. Hence~\eqref{taubounds} implies
\[
    0\leq\frac1n\log\left(\frac{{\det}^+(BB^T)}{\tau_\ell(\Delta)}\right)
    \leq\frac{h\log(kn)}n
    \leq\frac{k\log(kn)}{d+1}\longrightarrow0\,,
\]
using $\log n/d\to0$. This proves~\eqref{topologicalcomparison}. Finally,
\[
    \frac{\log{\det}^+(d^{-1}BB^T)}r
    =\frac{n}{r}\frac{\log(d^{-n}\tau_\ell(\Delta))}n
    +\left(\frac{n}{r}-1\right)\log d
    +\frac1r\log\left(\frac{{\det}^+(BB^T)}{\tau_\ell(\Delta)}\right)\,.
\]
By~\eqref{ranklognormalization2}, $n/r\to1$ and $(n/r-1)\log d\to0$. The last term tends to $0$ by~\eqref{topologicalcomparison}, so~\eqref{nicetau} and~\eqref{niceB} are equivalent. The same identity proves the equivalence in the graph case.
\end{proof}

\begin{corollary}\label{cor:torsion-spectral}
    Under the assumptions of Lemma~\ref{matrixtreeequivalence}, suppose that the measures $\mu_j$ satisfy~\eqref{niceB2}. Then
    \[
        \frac{\log| \widetilde H_{\ell-1}(\mathcal{T}_j,\mathbb{Z})|}{f_{\ell-1}(\Delta_j)}
        \text{ converges to }c_k\text{ in probability.}
    \]
    In particular, this holds if the smallest nonzero eigenvalue of $d_j^{-1}B_jB_j^T$ is bounded away from $0$ uniformly in $j$.
\end{corollary}

\begin{proof}
Corollary~\ref{thmBTB} shows that~\eqref{niceB2} implies~\eqref{niceB}. Lemma~\ref{matrixtreeequivalence} then gives~\eqref{nicetau}, and Theorem~\ref{thmBTB2} proves the statement.
\end{proof}

\begin{remark}\label{remarkdetvsHom}
For every spanning $\ell$-hyperforest $S$ of $\Delta_j$, Lemma~\ref{lemmaBTBvsP} and the determinantal weight formula give
\[
    \frac{\det(B_j^TB_j\restriction F_\ell(S))}{t_{\ell-1}(S)^2}
    =\frac{{\det}^+(B_jB_j^T)}{\tau_\ell(\Delta_j)}\,.
\]
Under the assumptions of Lemma~\ref{matrixtreeequivalence}, the logarithm of this ratio is $o(f_{\ell-1}(\Delta_j))$, uniformly over $S$. Thus, after normalization by $f_{\ell-1}(\Delta_j)$, the logarithm of the Gram determinant differs by $o(1)$ from twice the logarithm of the homology torsion.
\end{remark}

\medskip

\textbf{Acknowledgments.} The first author was supported by the Marie Sk\l{}odowska-Curie Postdoctoral Fellowship "RaCoCoLe".
The second author was partially supported by the Israel Science Foundation grant ISF-3464/24.

\medskip

\textbf{LLM disclosure statement. } The proof presented in Section~\ref{section1dcoupling} was found by ChatGPT Pro and written by the authors. AI assistance was also used in revising and extending the polytopal arguments in Section~\ref{SecTop}, checking mathematical details, proofreading, and improving the language.

\bibliography{references}

\bibliographystyle{plain}

\end{document}